\documentclass[a4paper, 11pt]{amsart}

\usepackage{amsmath, tabularx}
\usepackage{amsthm}
\usepackage{amssymb}
\usepackage{amsfonts}
\usepackage{paralist}
\usepackage{aliascnt}
\usepackage{amscd}
\usepackage{blkarray}
\usepackage{mathbbol}
\usepackage[inner=3.2cm,outer=3.2cm, bottom=3.5cm]{geometry}
\usepackage{setspace}
\usepackage[colorlinks=true,linkcolor=blue,urlcolor=blue]{hyperref}
\usepackage[initials]{amsrefs}
\usepackage{tikz}
\usetikzlibrary{matrix}
\usetikzlibrary{arrows}

\BibSpec{collection.article}{%
	+{}  {\PrintAuthors}                {author}
	+{,} { \textit}                     {title}
	+{.} { }                            {part}
	+{:} { \textit}                     {subtitle}
	+{,} { \PrintContributions}         {contribution}
	+{,} { \PrintConference}            {conference}
	+{}  {\PrintBook}                   {book}
	+{,} { }                            {booktitle}
	+{,} { }                            {series}
	+{, vol.} { }                            {volume}
	+{,} { }                            {publisher}
	+{,} { \PrintDateB}                 {date}
	+{,} { pp.~}                        {pages}
	+{,} { }                            {status}
	+{,} { \PrintDOI}                   {doi}
	+{,} { available at \eprint}        {eprint}
	+{}  { \parenthesize}               {language}
	+{}  { \PrintTranslation}           {translation}
	+{;} { \PrintReprint}               {reprint}
	+{.} { }                            {note}
	+{.} {}                             {transition}
	+{}  {\SentenceSpace \PrintReviews} {review}
}

\newcommand{\kk}{\mathbb{k}}
\newcommand{\KK}{\mathbb{K}}

\newcommand{\bmm}{\mathbb{M}}
\newcommand{\obmm}{{\overline{\bmm}}}
\newcommand{\NN}{\normalfont\mathbb{N}}
\newcommand{\ZZ}{\normalfont\mathbb{Z}}

\newcommand{\MM}{{\normalfont\mathfrak{M}}}
\newcommand{\PP}{{\normalfont\mathbb{P}}}

\newcommand{\XX}{{\normalfont\mathbf{X}}}
\newcommand{\bXX}{{\normalfont\mathbb{X}}}
\newcommand{\bYY}{{\normalfont\mathbb{Y}}}

\newcommand{\dd}{{\normalfont\mathbf{d}}}

\newcommand{\mm}{{\normalfont\mathfrak{m}}}

\newcommand{\QQ}{\mathbb{Q}}
\newcommand{\pp}{{\normalfont\mathfrak{p}}}

\newcommand{\nn}{{\normalfont\mathfrak{N}}}
\newcommand{\bn}{{\normalfont\mathbf{n}}}

\newcommand{\bm}{{\normalfont\mathbf{m}}}
\newcommand{\ttt}{{\normalfont\mathbf{t}}}
\newcommand{\nnn}{{\normalfont\mathfrak{n}}}

\newcommand{\depth}{\normalfont\text{depth}}

\newcommand{\Tor}{\normalfont\text{Tor}}

\newcommand{\pd}{\normalfont\text{pd}}
\newcommand{\Ker}{\normalfont\text{Ker}}

\newcommand{\HT}{\normalfont\text{ht}}

\newcommand{\Ann}{\normalfont\text{Ann}}
\newcommand{\Supp}{\normalfont\text{Supp}}
\newcommand{\Ass}{{\normalfont\text{Ass}}}

\newcommand{\Rees}{\mathcal{R}}
\newcommand{\ee}{{\normalfont\mathbf{e}}}

\newcommand{\Fitt}{\normalfont\text{Fitt}}
\newcommand{\sF}{\widetilde{\mathfrak{F}}}

\newcommand{\JJ}{\mathcal{J}}

\newcommand{\BB}{\mathfrak{B}}

\newcommand{\EE}{\mathbb{E}}

\newcommand{\OO}{\mathcal{O}}

\newcommand{\FF}{\normalfont\mathcal{F}}

\newcommand{\HH}{\normalfont\text{H}}

\newcommand{\gr}{{\normalfont\text{gr}}}

\newcommand{\AAA}{\mathfrak{A}}

\newcommand{\bideg}{\normalfont\text{bideg}}
\newcommand{\Proj}{\normalfont\text{Proj}}
\newcommand{\Hilb}{{\normalfont\text{Hilb}}}
\newcommand{\Spec}{\normalfont\text{Spec}}

\newcommand{\multProj}{\normalfont\text{MultiProj}}
\newcommand{\biProj}{{\normalfont\text{BiProj}}}

\newtheorem{theorem}{Theorem}[section]

\newtheorem{headthm}{Theorem}

\newaliascnt{headcor}{headthm}

\aliascntresetthe{headcor}

\newaliascnt{headthmdef}{headthm}
\newtheorem{headthmdef}[headthmdef]{Definition-Theorem}
\aliascntresetthe{headthmdef}

\newaliascnt{headconj}{headthm}

\aliascntresetthe{headconj}

\newaliascnt{corollary}{theorem}

\aliascntresetthe{corollary}

\newaliascnt{lemma}{theorem}
\newtheorem{lemma}[lemma]{Lemma}
\aliascntresetthe{lemma}

\newaliascnt{conjecture}{theorem}

\aliascntresetthe{conjecture}

\newaliascnt{proposition}{theorem}
\newtheorem{proposition}[proposition]{Proposition}
\aliascntresetthe{proposition}

\theoremstyle{definition}
\newaliascnt{definition}{theorem}
\newtheorem{definition}[definition]{Definition}
\aliascntresetthe{definition}

\newaliascnt{notation}{theorem}

\aliascntresetthe{notation}

\newaliascnt{example}{theorem}

\aliascntresetthe{example}

\newaliascnt{examples}{theorem}

\aliascntresetthe{examples}

\newaliascnt{remark}{theorem}
\newtheorem{remark}[remark]{Remark}
\aliascntresetthe{remark}

\newaliascnt{problem}{theorem}

\aliascntresetthe{problem}

\newaliascnt{construction}{theorem}

\aliascntresetthe{construction}

\newaliascnt{setup}{theorem}
\newtheorem{setup}[setup]{Setup}
\aliascntresetthe{setup}

\newaliascnt{algorithm}{theorem}

\aliascntresetthe{algorithm}

\newaliascnt{observation}{theorem}

\aliascntresetthe{observation}

\newaliascnt{defprop}{theorem}

\aliascntresetthe{defprop}

\def\equationautorefname~#1\null{(#1)\null}
\def\sectionautorefname~#1\null{Section #1\null}
\def\subsectionautorefname~#1\null{\S #1\null}

\def\surjects{\twoheadrightarrow}

\begin{document}

\title{Mixed multiplicities and projective degrees of rational maps}

\author{Yairon Cid-Ruiz}
\address[Cid-Ruiz]{Max Planck Institute for Mathematics in the Sciences, Inselstra\ss e 22, 04103 Leipzig, Germany.}
\email{cidruiz@mis.mpg.de}
\urladdr{https://ycid.github.io}


\begin{abstract}
	We consider the notion of mixed multiplicities for multigraded modules by using Hilbert series, and this is later applied to study the projective degrees of rational maps. 
	We use a general framework to determine the projective degrees of a rational map via a computation of the multiplicity of  the saturated special fiber ring. 
	As specific applications, we provide explicit formulas for all the projective degrees of rational maps determined by perfect ideals of height two or by Gorenstein ideals of height three.
\end{abstract}

\maketitle

\section{Introduction}

Let $\BB$ be a standard multigraded algebra $\BB=\bigoplus_{\nu \in \NN^r} {\left[\BB\right]}_\nu$ over an Artinian local ring $A={\left[\BB\right]}_{(0,\ldots,0)}$.
Let $\bmm$ be a finitely generated $\ZZ^r$-graded $\BB$-module.
It is known that the Hilbert function 
$$
\nu \in \ZZ^r \;\, \mapsto \;\, \text{length}_A\big({\left[\bmm\right]}_\nu\big) \in \NN
$$
of $\bmm$ coincides with a polynomial $P_\bmm(\XX)=P_\bmm(X_1,\ldots,X_r)$ for $\nu \gg (0,\ldots,0) \in \NN^r$ (see \cite[Theorem 4.1]{HERMANN_MULTIGRAD}).
This polynomial $P_\bmm(\XX)$ is referred to as the \textit{Hilbert polynomial} of $\bmm$, and its total degree is equal to the dimension $d_{++}=\dim\left(\Supp_{++}(\bmm)\right)$ of the relevant support $\Supp_{++}(\bmm)$ of $\bmm$.
Let $\nn \subset \BB$ be the multigraded irrelevant ideal $\nn = \bigoplus_{\nu_1>0,\ldots,\nu_r >0}{\left[\BB\right]}_\nu$.
Then, the relevant support $\Supp_{++}(\bmm)$ is given by $\Supp_{++}(\bmm)=\lbrace \pp \in \Supp(\bmm) \mid \pp \text{ is $\NN^r$-graded and } \pp \not\supseteq \nn \rbrace$.
If we write 
$$
P_{\bmm}(\XX) = \sum_{n_1,\ldots,n_r \ge 0} e(n_1,\ldots,n_r)\binom{X_1+n_1}{n_1}\cdots \binom{X_r+n_r}{n_r},
$$
then $0 \le e(n_1,\ldots,n_r) \in \ZZ$ for all $n_1+\ldots+n_r = d_{++}$.
For $\bn \in \NN^r$ with $\lvert \bn \rvert = n_1+\cdots+n_r=d_{++}$, the non-negative integer $e(n_1,\ldots,n_r)$ is called the \textit{mixed multiplicity of $\bmm$ of type $\mathbf{n}=(n_1,\ldots,n_r)$} and it is denoted as $e(\mathbf{n};\bmm)=e(n_1,\ldots,n_r;\bmm)$.

The concept of mixed multiplicities provides the right generalization of multiplicities (or degrees) to a multigraded setting, and its study goes back to seminal work by van der Waerden \cite{VAN_DER_WAERDEN}.
These invariants have been further studied and developed by Bhattacharya \cite{Bhattacharya}, by Katz, Mandal and Verma \cite{VERMA_BIGRAD}, by Herrmann, Hyry, Ribbe and Tang \cite{HERMANN_MULTIGRAD}, and by Trung \cite{TRUNG_POSITIVE}. 
For more details, see the survey paper \cite{TRUNG_VERMA_SURVEY} and the references therein.

Motivated by the notion of multidegree defined by Miller and Sturmfels in \cite[\S 8.5]{MILLER_STURMFELS}, in this paper,  we develop the theory of mixed multiplicities for multigraded modules by using \textit{Hilbert series}. 
Contrary to the single-graded case, in a multigraded setting the approaches with Hilbert polynomials or Hilbert series may yield different results. 
This comes from the fact that the Hilbert polynomial can only read irreducible components which are relevant in the multigraded geometric sense. 
To be more precise, let $d = \dim(\bmm)$ and denote the Hilbert series of $\bmm$ by 
$$
\Hilb_\bmm(\ttt) = \Hilb_\bmm(t_1,\ldots,t_r) = \sum_{\nu \in \ZZ^r} \,\text{length}_A\left({\left[\bmm\right]}_\nu\right)\, t_1^{\nu_1}\cdots t_r^{\nu_r}.
$$
In our first main result we provide a different notion of mixed multiplicities (defined in terms of Hilbert series) and we relate it to the above notion of mixed multiplicities (defined in terms of Hilbert polynomials).
For the moment, note that $\dim\left(\bmm/\HH_{\nn}^0(\bmm)\right)=r+d_{++}$ (see \autoref{lem_properties_quot_torsion}).
We show that the mixed multiplicities $e(\bn;\bmm)$ can be expressed in terms of the new ones applied to the module $\bmm/\HH_\nn^0(\bmm)$, i.e., after modding out the torsion with respect to $\nn$.

\medskip

\begin{headthmdef}[\autoref{thm_struc_Hilb_series}, \autoref{lem_unique_values}, \autoref{def_mixed_mult}, \autoref{thm_equality_mults}]
	Let $\BB$ be a standard multigraded algebra $\BB=\bigoplus_{\nu \in \NN^r} {\left[\BB\right]}_\nu$ over an Artinian local ring.
	Let $\bmm$ be a finitely generated $\ZZ^r$-graded $\BB$-module.
	Let $d=\dim(\bmm)$ and $d_{++}=\dim\left(\Supp_{++}(\bmm)\right)$.	
	\begin{enumerate}[(I)]
		
		\medskip
		\item[(I) {\rm [A structural result for Hilbert series]}] For each $\bn = (n_1,\ldots,n_r) \in \NN^r$ with $\vert \bn \rvert = d$, there exists a Laurent polynomial $Q_\bn(\ttt) = Q_{(n_1,\ldots,n_r)}(t,\ldots,t) \in \ZZ[\ttt,\ttt^{\mathbf{-1}}]=\ZZ[t_1,\ldots,t_r,t_1^{-1},\ldots,t_r^{-1}]$ with $Q_\bn(\mathbf{1})=Q_{(n_1,\ldots,n_r)}(1,\ldots,1) \ge 0$, such that
		$$
			\Hilb_\bmm(\ttt) = \sum_{\lvert\bn\rvert = d} \frac{Q_\bn(\ttt)}{(1-t_1)^{n_1}\cdots(1-t_r)^{n_r}},
		$$
		and the following statements hold:
		\begin{enumerate}[(a)]
			\item There is at least one $\bn \in \NN^r$ with $\lvert \bn \rvert = d$ such that $Q_\bn(\mathbf{1}) > 0$.
			\item The values of $Q_\bn(\mathbf{1})$ are uniquely determined by the module $\bmm$.
		\end{enumerate}
		
		\medskip
		\item[(II) {\rm [Definition]}] For each $\bn = (n_1,\ldots,n_r) \in \ZZ^r$  with $n_i \ge -1$ and $\lvert \bn \rvert = d-r$, we define {\tt the mixed multiplicity of $\bmm$ of type $\bn = (n_1,\ldots,n_r)$ in terms of Hilbert series} as 
		$$
		e_\bn\left(\bmm\right) = e_{\left(n_1,\ldots,n_r\right)}\left(\bmm\right) = Q_{\mathbf{n+1}}(\mathbf{1}),
		$$
		where $Q_{\mathbf{n+1}}(\ttt)=Q_{\left(n_1+1,\ldots,n_r+1\right)}(t,\ldots,t) \in \ZZ[\ttt,\ttt^{\mathbf{-1}}]$ are Laurent polynomials obtained in part $(I)$.
		
		\medskip
		\item[(III) {\rm [Relation between the two notions of mixed multiplicities]}] For each $\bn \in \NN^r$ with $\lvert \bn \rvert = d_{++}$,  we have the following equality
		$$
		e(\bn;\bmm) = e_\bn\Big(\bmm/\HH_\nn^0(\bmm)\Big).
		$$		
	\end{enumerate}
\end{headthmdef}

\medskip

One advantage of our approach with Hilbert series is that we can read certain mixed multiplicities which cannot be read using the Hilbert polynomial approach.
For instance, if $\bmm$ has a minimal prime of maximal dimension containing the irrelevant ideal $\nn$, then the contribution of that minimal prime is summed up in some $e_\bn(\bmm)$ but not in any $e(\bn;\bmm)$.

After defining the new mixed multiplicities, we prove several general results about these invariants (see \autoref{lem_basic_props},  \autoref{thm_mult_gr_mixed} and \autoref{thm_mixed_mult_as_single_grad}).
We also provide a different proof for the existence of the Hilbert polynomial in a multigraded setting (see \autoref{thm_equality_mults}). 
Similarly to \cite{TRUNG_POSITIVE}, we use \emph{filter-regular elements} to substitute the notion of \emph{general elements}.

\medskip

In the second half of the paper, we consider the multidegrees of multiprojective schemes as mixed multiplicities (see \autoref{def_multdegree}). 
In particular, our main focus will be on the \textit{projective degrees} of rational maps.
The projective degrees of rational maps are fundamental and classical invariants in Algebraic Geometry.
For more details on the subject, the reader is referred to \cite[Example 19.4]{HARRIS} and \cite[\S 7.1.3]{DOLGACHEV}.
Recently, the projective degrees of rational maps have also been considered in the area of Algebraic Statistics in \cite[\S 5]{EXPONENTIAL_VARIETIES}.

\medskip

Next, we describe our results regarding projective degrees.
Let $\kk$ be a field and $R$ be the polynomial ring $R=\kk[x_0,\ldots,x_d]$.
Let $n \ge d$ and 
\begin{equation}
	\label{eq_rat_map}
	\FF : \PP_\kk^d = \Proj(R) \dashrightarrow  \PP_\kk^n
\end{equation}
be a rational map defined by $n+1$ homogeneous elements $\{f_0,  \ldots, f_n\} \subset R$ of the same degree $\delta > 0$ and $I \subset R$ be the homogeneous ideal $I=\left(f_0,\ldots,f_n\right)$.
The projective degrees of $\FF$ are defined as the multidegrees of the graph of $\FF$ (see \autoref{def_proj_degrees}), and they are denoted as $d_i(\FF)$ for $0 \le i \le d$.
In classical geometrical terms, $d_i(\FF)$ equals the number of points in the intersection of the graph $\Gamma \subset \PP_\kk^d \times_\kk \PP_\kk^n$ of $\FF$ with the product $H \times_\kk K \subset \PP_\kk^d \times_\kk \PP_\kk^n$, where $H \subset \PP_\kk^d$ and $K \subset \PP_\kk^n$ are general subspaces of dimension $d-i$ and $n-d+i$, respectively.
 
Our main tool for the computation of the projective degrees $d_i(\FF)$ will be to exploit the \textit{saturated special fiber ring} \cite{MULTPROJ} (see \autoref{thm_equal_d_0_sat_fib}).

As specific applications, we compute all the projective degrees for rational maps determined by  perfect ideals of height two or by Gorenstein ideals of height three.
In both cases we assume the condition $G_{d+1}$, that is, $\mu(I_{\pp}) \le \dim(R_{\pp})$  for all $\pp \in V(I) \subset \Spec(R)$  such that $\HT(\pp)<d+1$.  However, it should be noted that the condition $G_{d+1}$ is always satisfied by generic perfect ideals of height two and by generic Gorenstein ideals of height three.

\medskip

In the conditions expressed in the theorem below, we use the fact that the minimal resolution of any perfect ideal of height two is described by the Hilbert-Burch theorem (see, e.g., \cite[Theorem 20.15]{EISEN_COMM}).

 \begin{headthm}[\autoref{thm_proj_deg_perf_ht_2}]
 		\label{thmB}
 		With the notations above, assume the following conditions:
 	\begin{enumerate}[(i)]
 		\item $I$ is perfect of height two with Hilbert-Burch resolution of the form
 		$$
 		0 \rightarrow \bigoplus_{i=1}^nR(-\delta-\mu_i) \xrightarrow{\varphi} {R(-\delta)}^{n+1} \rightarrow I\rightarrow 0.			
 		$$
 		\item $I$ satisfies the condition $G_{d+1}$.
 	\end{enumerate}
 	Then, the projective degrees of the rational map $\FF:\PP_\kk^d \dashrightarrow \PP_\kk^n$ in \autoref{eq_rat_map} are given by
 	$$
 	d_i(\FF) =  
 	e_{d-i}(\mu_1,\mu_2,\ldots,\mu_n)
 	$$
 	where $	e_{d-i}(\mu_1,\mu_2,\ldots,\mu_n)$ denotes the elementary symmetric polynomial
 	$$
 	e_{d-i}(\mu_1,\mu_2,\ldots,\mu_n)= 
 	\sum_{1\le j_1 < j_2 < \cdots < j_{d-i} \le n} \mu_{j_1}\mu_{j_2}\cdots\mu_{j_{d-i}}.
 	$$
 \end{headthm}

\smallskip

In the theorem below, we use the fact that the minimal resolution of any Gorenstein ideal of height three is described by the Buchsbaum-Eisenbud structure theorem \cite{BuchsbaumEisenbud}.

\smallskip

\begin{headthm}[\autoref{thm_proj_deg_Gor_ht_3}]
	\label{thmC}
	With the notations above, assume the following conditions:
	\begin{enumerate}[(i)]
		\item $I$ is a Gorenstein ideal of height three.
		\item Every non-zero entry of an alternating minimal presentation matrix of $I$  has degree $D\ge 1$.
		\item $I$ satisfies the condition $G_{d+1}$.		
	\end{enumerate}
	Then, the projective degrees of the rational map $\FF:\PP_\kk^d \dashrightarrow \PP_\kk^n$ in \autoref{eq_rat_map} are given by
	$$
	d_i(\FF) = \begin{cases}
	D^{d-i} \sum_{k=0}^{\lfloor\frac{n-d+i}{2}\rfloor}\binom{n-1-2k}{d-i-1} \quad \text{ if } 0 \le i \le d-3 \\
	\delta^{d-i} \quad\quad\quad\;\,\,\,\,\qquad\qquad\qquad \text{if } d-2 \le i \le d.
	\end{cases}
	$$	
\end{headthm}

\medskip

The basic outline of this paper is as follows.
In \autoref{sect2}, we introduce the notion of mixed multiplicities in terms of Hilbert series. 
In \autoref{sect3}, we relate this notion with the usual mixed multiplicities in terms of Hilbert polynomials. 
In \autoref{sect4}, we consider the multidegrees of multiprojective schemes and we prove van der Waerden's original result \cite{VAN_DER_WAERDEN} in a multigraded setting.
In \autoref{sect5}, we concentrate on the projective degrees of rational maps and we obtain the formulas of \autoref{thmB} and \autoref{thmC}.

\section{Mixed multiplicities of multigraded modules via Hilbert series}
\label{sect2}

During this section, we study and develop a definition of mixed multiplicities that depends on Hilbert series. 
The results in this section can be seen as a natural continuation of \cite[\S 8.5]{MILLER_STURMFELS}; in particular, our approach does not depend on the existence of finite free resolutions. 
We begin by fixing some notation and recalling some basic facts.

We use a multi-index notation.
For any $\nu \in \ZZ^r$, we define its weight as $\lvert \nu \rvert= \nu_1+\cdots+\nu_r$. 
If $\mu,\nu \in \ZZ^r$ are two multi-indexes, we write $\mu \ge \nu$ whenever $\mu_i \ge \nu_i$, and $\mu > \nu$ whenever $\mu_i > \nu_i$.
For $1 \le i \le r$, let $\ee_i$ be the $i$-th elementary vector $\ee_i=\left(0,\ldots,1,\ldots,0\right)$.
Let $\mathbf{0} \in \NN^r$ and $\mathbf{1} \in \NN^r$ be the vectors $\mathbf{0}=(0,\ldots,0)$ and $\mathbf{1}=(1,\ldots,1)$ of $r$ copies of $0$ and $1$, respectively.

The following setup is fixed throughout this section. 

\begin{setup}
	\label{setup_general}
	Let $(A,\nnn,\kk)$ be an Artinian local ring with maximal ideal $\nnn$ and residue field $\kk=A/\nnn$. 
	Let $\BB$ be a finitely generated standard $\NN^r$-graded algebra over $A$, that is,  $\left[\BB\right]_{\mathbf{0}}=A$ and $\BB$ is finitely generated over $A$ by elements of degree $\ee_i$ with $1 \le i \le r$.
\end{setup} 

For any $\ZZ^r$-graded $\BB$-module $\bmm$, its graded part of degree $\nu \in \ZZ^r$ is denoted by $\left[\bmm\right]_\nu$.
Unless specified otherwise, a graded $\BB$-module $\bmm$ will always means an $\ZZ^r$-graded $\BB$-module, that is, $\bmm=\bigoplus_{\nu \in \ZZ^r} \left[\bmm\right]_\nu$.
If $\nu \in \ZZ^r$ and $x \in \left[\BB\right]_\nu$, we say that $x$ is homogeneous of degree $\nu$ and its total degree is $\lvert \nu \rvert$.
Let $\MM \subset \BB$ be the graded ideal $\MM := \bigoplus_{\nu\neq \mathbf{0}} \left[\BB\right]_\nu$, and for $1 \le i \le r$ let $\MM_i \subset \BB$ be the graded ideal $\MM_i := \bigoplus_{\nu\ge \ee_i} \left[\BB\right]_\nu$.
In this multigraded setting the irrelevant ideal is defined as the graded ideal
$$
\nn := \MM_1 \cap \cdots \cap \MM_r = \bigoplus_{\nu > \mathbf{0}} \left[\BB\right]_\nu.
$$

Similarly to the single-graded case, we can define a multiprojective scheme from $\BB$.
The multiprojective scheme  $\multProj(\BB)$ is given by 
the set of all graded prime ideals in $\BB$ which do not contain $\nn$, that is,
$$
\multProj(\BB) := \big\{ \pp \in \Spec(\BB) \mid \pp \text{ is graded and } \pp \not\supseteq \nn \big\},
$$
and its scheme structure is obtained by using multi-homogeneous localizations (see, e.g., \cite[\S 1]{HYRY_MULTIGRAD}). 
The closed subsets of  $\multProj(\BB)$ are given by $V_{++}(\JJ):=V(\JJ)\cap \multProj(\BB) = \lbrace \pp \in \multProj(\BB) \mid \pp \supseteq \JJ \rbrace$ for $\JJ \subset \BB$ a graded ideal.
For a finitely generated graded $\BB$-module $\bmm$, set $\Supp_{++}\left(\bmm\right)$ to be the closed subset $\Supp_{++}\left(\bmm\right):=\Supp(\bmm) \cap \multProj(\BB)=\lbrace \pp \in \multProj(\BB) \mid \bmm_\pp \neq 0\rbrace=V_{++}(\Ann(\bmm))$.
Given a graded $\BB$-module $\bmm$, we denote as $\bmm^\gr$ the single-graded module given as
$$
\bmm^\gr := \bigoplus_{n \in \ZZ} \left(\bigoplus_{\substack{\nu \in \ZZ^r\\ \lvert \nu \rvert = n}} \left[\bmm\right]_\nu\right).
$$
We have that $\bmm^\gr$ is naturally a graded module over the single-graded $A$-algebra $\BB^\gr$.

For any $\bn \in \ZZ^r$, the terms $\ttt^\bn$ and $(\mathbf{1-t})^\bn$ represent the elements $t_1^{n_1}\cdots t_r^{n_r}$ and $(1-t_1)^{n_1}\cdots (1-t_r)^{n_r}$, respectively.
Note that any graded part $\left[\bmm\right]_\nu$ of a finitely generated graded $\BB$-module $\bmm$ is a finitely generated module over the Artinian local ring $A$, and so a module of finite length.
The length of a finitely generated $A$-module $E$ is denoted as $\text{length}_A(E)$.
The Hilbert series of a finitely generated graded $\BB$-module $\bmm$ is denoted as $\Hilb_\bmm(\ttt)$ and given by 
$$
\Hilb_\bmm(\ttt) := \sum_{\nu \in \ZZ^r} \text{length}_A\big(\left[\bmm\right]_\nu\big) \ttt^\nu.
$$
For any $\nu \in \ZZ^r$, $\bmm(\nu)$ denotes the shifted graded $\BB$-module given as $\left[\bmm(\nu)\right]_\mu = \left[\bmm\right]_{\nu+\mu}$.
Also, we have that $\Hilb_{\bmm(-\nu)}(\ttt)=\ttt^\nu\Hilb_\bmm(\ttt)$.

The following theorem gives a structural result for the Hilbert series of a finitely generated graded $\BB$-module.

\begin{theorem}
	\label{thm_struc_Hilb_series}
	Assume \autoref{setup_general}.
	Let $\bmm$ be a finitely generated graded $\BB$-module.
	Let $d=\dim(\bmm)$.
	Then, for each $\bn \in \NN^r$ with $\vert \bn \rvert = d$, there exists a Laurent polynomial $Q_\bn(\ttt) \in \ZZ[\ttt,\ttt^{\mathbf{-1}}]$ with $Q_\bn(\mathbf{1}) \ge 0$, such that
	$$
	\Hilb_\bmm(\ttt) = \sum_{\lvert\bn\rvert = d} \frac{Q_\bn(\ttt)}{(\mathbf{1-t})^\bn},
	$$
	and the following statements hold:
	\begin{enumerate}[(i)]
		\item There is at least one $\bn \in \NN^r$ with $\lvert \bn \rvert = d$ such that $Q_\bn(\mathbf{1}) > 0$.
		\item If $Q_\bn(\ttt)\neq 0$ and $Q_\bn(\mathbf{1})=0$, then $Q_\bn(\ttt)$ is divisible by $(1-t_i)$ for some $1 \le i \le r$ such that $n_i\neq 0$.
	\end{enumerate}
	
	\begin{proof}
		We proceed by induction on $d$.
		
		For the statement in part $(ii)$, we will actually prove that, if $Q_\bn(\ttt)\neq 0$ and $Q_\bn(\mathbf{1})=0$, then $\bn \ge \ee_1$ and $(1-t_1)$ divides $Q_\bn(\ttt)$.
		
		There exists a finite filtration 
		$$
		0 = \bmm_0 \subset \bmm_1 \subset \cdots \subset \bmm_k=\bmm
		$$
		of $\bmm$ such that $\bmm_l/\bmm_{l-1} \cong \left(\BB/\pp_l\right)(-\nu_l)$ where $\pp_l \subset \BB$ is a graded prime ideal with dimension $\dim(\BB/\pp_l)\le d$ and $\nu_l \in \ZZ^r$.
		The short exact sequences
		\begin{equation*}
		0 \rightarrow \bmm_{l-1} \rightarrow \bmm_l \rightarrow \left(\BB/\pp_l\right)(-\nu_l) \rightarrow 0		
		\end{equation*}
		and the clear additivity of Hilbert series yield that 
		$$
		\Hilb_\bmm(\ttt) = \sum_{l=1}^k \ttt^{\nu_l} \Hilb_{\BB/\pp_l}(\ttt).
		$$
		Let $d_l=\dim\left(\BB/\pp_l\right)$ and suppose for the moment that we have shown that 
		$$
		\Hilb_{\BB/\pp_l}(\ttt) = \sum_{\lvert\bn\rvert = d_l} \frac{Q_\bn^{(l)}(\ttt)}{(\mathbf{1-t})^\bn},
		$$
		where $Q_\bn^{(l)}(\ttt) \in \ZZ[\ttt,\ttt^{\mathbf{-1}}]$ satisfy the same conditions of parts $(i)$ and $(ii)$ (with $i=1$).
		By an abuse of notation, for any $\bm \not\ge \mathbf{0}$ we set $Q_{\bm}^{(l)}(\ttt)=0$.
		Then, we would get the equation
		\begin{align*}
			\begin{split}
			\Hilb_\bmm(\ttt) &= \sum_{l=1}^k \ttt^{\nu_l}\frac{(1-t_1)^{d-d_l}}{(1-t_1)^{d-d_l}} \Hilb_{\BB/\pp_l}(\ttt)	 \\
			&= 	\sum_{\lvert\bn\rvert = d} 
			\frac{
				\sum_{l=1}^k\ttt^{\nu_l}(1-t_1)^{d-d_l}Q_{\bn-(d-d_l)\ee_1}^{(l)}(\ttt)
			}
			{(\mathbf{1-t})^\bn}.
			\end{split}			
		\end{align*}		
		
		Therefore, it is enough to consider the case $\bmm=\BB/\pp$ where $\pp \subset \BB$ is a graded prime ideal.		
		Suppose that $d=\dim\left(\BB/\pp\right)=0$.
		Since $A$ is an Artinian local ring, it follows that $\pp$ is equal to the unique maximal graded ideal $\nnn + \MM$.
		Hence, it follows that $\Hilb_{\BB/\pp}(\ttt)=1$, and the result is clear.
		
		Suppose that $d=\dim\left(\BB/\pp\right)>0$.
		Then, we may choose a homogeneous element $0 \neq x \in \BB/\pp$ of degree $\deg(x)=\ee_i$ where $1 \le i \le r$.
		From the short exact sequence
		$$
		0 \rightarrow \left(\BB/\pp\right)(-\ee_i) \xrightarrow{x} \BB/\pp \rightarrow \BB/(x,\pp) \rightarrow 0
		$$
		we obtain the equation
		$$
		\Hilb_{\BB/(x,\pp)}(\ttt) = \Hilb_{\BB/\pp}(\ttt)-t_i\Hilb_{\BB/\pp}(\ttt)=(1-t_i)\Hilb_{\BB/\pp}(\ttt).
		$$
		Since $\dim\left(\BB/(x,\pp)\right)=d-1$, from the induction hypothesis we may assume that 
		$$
		\Hilb_{\BB/(x,\pp)}(\ttt) = \sum_{\lvert\bn\rvert = d-1} \frac{Q_\bn(\ttt)}{(1-t_1)^{n_1}\cdots(1-t_r)^{n_r}},
		$$
		where $Q_\bn(\ttt) \in \ZZ[\ttt,\ttt^{\mathbf{-1}}]$ satisfy the same conditions of parts $(i)$ and $(ii)$ (with $i=1$).
		Thus, the equation $\Hilb_{\BB/\pp}(\ttt) = \frac{1}{1-t_i}\Hilb_{\BB/(x,\pp)}(\ttt)$ implies that 
		$$
		\Hilb_{\BB/\pp}(\ttt) = \sum_{\lvert\bn\rvert = d-1} \frac{Q_\bn(\ttt)}{(1-t_1)^{n_1}\cdots(1-t_i)^{n_i+1}\cdots(1-t_r)^{n_r}},
		$$
		and so the statement of the theorem is obtained.		
	\end{proof}
\end{theorem}

Although the decomposition of the Hilbert series $\Hilb_\bmm(\ttt)$ given in \autoref{thm_struc_Hilb_series} is not necessarily unique, we can easily see from the simple lemma below that the values of $Q_\bn(\mathbf{1})$ are uniquely determined by the module $\bmm$.

\begin{lemma}
	\label{lem_unique_values}
	Suppose that 
	$$
	\sum_{\lvert\bn\rvert = d} \frac{Q_\bn(\ttt)}{(\mathbf{1-t})^\bn} = \sum_{\lvert\bn\rvert = d} \frac{Q_\bn^\prime(\ttt)}{(\mathbf{1-t})^\bn}
	$$
	where $\bn \in \NN^r$ and $Q_\bn(\ttt), Q_\bn^\prime(\ttt) \in \ZZ[\ttt,\ttt^{\mathbf{-1}}]$. 
	Then, we have that $Q_\bn(\mathbf{1})=Q_\bn^\prime(\mathbf{1})$ for all $\bn \in \NN^r$ with $\lvert \bn \rvert = d$.
	\begin{proof}
		Take $\bm \in \NN^r$ such that $P_\bn(\ttt)=\ttt^\bm Q_\bn(\ttt) \in \ZZ[\ttt]$ and $P_\bn^\prime(\ttt)=\ttt^\bm Q_\bn^\prime(\ttt) \in \ZZ[\ttt]$, i.e., $P_\bn(\ttt)$ and $P_\bn^\prime(\ttt)$ become polynomials in $\ttt$.
		Multiplying both sides of the above equation by $(\mathbf{1-t})^\dd\ttt^\bm$ (where $\dd=(d,d,\ldots,d) \in \NN^r$) gives us that 
		$$
		\sum_{\lvert\bn\rvert = d} P_\bn(\ttt)(\mathbf{1-t})^{\dd-\bn} = \sum_{\lvert\bn\rvert = d} P_\bn^\prime(\ttt)(\mathbf{1-t})^{\dd-\bn}.
		$$
		Then, the substitution $\ttt \mapsto \mathbf{1-t}$ yields 
		$$
		\sum_{\lvert\bn\rvert = d} P_\bn(\mathbf{1-t})\ttt^{\dd-\bn} = \sum_{\lvert\bn\rvert = d} P_\bn^\prime(\mathbf{1-t})\ttt^{\dd-\bn}.
		$$
		Since both sums above are restricted to the multi-indexes $\bn \in \NN^r$ with $\lvert \bn \rvert=d$, by comparing the terms of the smallest possible degree $\lvert \dd -\bn \rvert=(r-1)d$, we obtain that $Q_\bn(\mathbf{1})=P_\bn(\mathbf{1})=P_\bn^\prime(\mathbf{1})=Q_\bn^\prime(\mathbf{1})$, and so the result follows.
	\end{proof}
\end{lemma}

We are now ready to define the following notion of mixed multiplicities in a multigraded setting.

\begin{definition}
	\label{def_mixed_mult}
	Assume \autoref{setup_general}. 
	Let $\bmm$ be a finitely generated graded $\BB$-module. 
	Let $d=\dim(\bmm)$. 
	From \autoref{thm_struc_Hilb_series} choose any decomposition 
	$$
	\Hilb_\bmm(\ttt) = \sum_{\substack{\bn \in \NN^r\\\lvert\bn\rvert = d}} \frac{Q_\bn(\ttt)}{(\mathbf{1-t})^\bn}.
	$$
	For any $\bn=(n_1,\ldots,n_r) \in \ZZ^r$ with $\bn \ge \mathbf{-1} = (-1,\ldots,-1) \in \ZZ^r$ and $\lvert \bn+\mathbf{1} \rvert \ge d$, the \textit{mixed multiplicity of $\bmm$ of type $\bn$ defined in terms of Hilbert series} is given by 
	$$
	e_\bn\left(\bmm\right) := \begin{cases}
		Q_{\bn+\mathbf{1}}(\mathbf{1}) \quad\text{ if } \lvert \bn + \mathbf{1} \rvert = d\\
		0 \qquad\qquad\,\,\; \text{if } \lvert \bn + \mathbf{1} \rvert > d.
	\end{cases} 
	$$	
	From \autoref{lem_unique_values}, the mixed multiplicities $e_\bn\left(\bmm\right)$ are uniquely determined by $\bmm$.
\end{definition}

\begin{remark}
	In the above definition we have chosen to enumerate the mixed multiplicities for multi-indexes with $\bn \ge \mathbf{-1}$.
	We made this choice so that we can relate $e_\bn(\bmm)$ with the usual definition in terms of Hilbert polynomials for $\bn\ge \mathbf{0}$ (see \autoref{thm_equality_mults} below).
	
	For instance, let $\KK$ be a field and $S=\KK\left[x_1,\ldots,x_n,y_1,\ldots,y_m\right]$ be a bigraded polynomial ring with $\deg(x_i)=(1,0)$ and $\deg(y_i)=(0,1)$.
	Then, the Hilbert polynomial and the Hilbert series of $S$ are given by 
	$$
	P_S(\nu_1,\nu_2)=\binom{\nu_1+n-1}{n-1}\binom{\nu_2+m-1}{m-1}
	$$
	and
	$$
	\Hilb_S(t_1,t_2) = \frac{1}{(1-t_1)^n(1-t_2)^m},
	$$
	respectively.
	Therefore, from both definitions, we obtain that $e_{n-1,m-1}(S)=e(n-1,m-1;S)=1$ and that $e_{i,j}(S)=e(i,j;S)=0$ for $i,j\ge0,\; i+j=n+m-2$.
\end{remark}

\begin{remark}
	In \autoref{def_mixed_mult} we allow the flexibility of having $\lvert \bn + \mathbf{1}\rvert \ge d$ so that the function $e_\bn(\bullet)$ becomes additive in the full subcategory of finitely generated graded $\BB$-modules with dimension at most $\lvert \bn + \mathbf{1} \rvert$ (for the same setting in the single-graded case, see, e.g., \cite[Corllary 4.7.7]{BRUNS_HERZOG}). 
\end{remark}

Next we derive some basic properties of the mixed multiplicities $e_\bn(\bullet)$.  

\begin{lemma}
	\label{lem_basic_props}
	Let $\bmm$ be a finitely generated graded $\BB$-module with $\dim(\BB)=d$.
	Let $\bn \in \ZZ^r$ such that $\bn \ge \mathbf{-1}$ and $\vert \bn + \mathbf{1} \rvert \ge d$.
	Then, the following statements hold:
	\begin{enumerate}[(i)]
		\item For any $\nu \in \ZZ$, we have that $e_\bn\left(\bmm(-\nu)\right)=e_\bn(\bmm)$.
		\item (additivity) Let $0 \rightarrow \bmm^\prime \rightarrow \bmm \rightarrow \bmm^{\prime\prime} \rightarrow 0$ be a short exact sequence of finitely generated graded $\BB$-modules.
		Then
		$$
		e_\bn(\bmm) = e_\bn(\bmm^{\prime}) + e_\bn(\bmm^{\prime\prime}).
		$$
		\item (associativity formula) 
		$$
		e_\bn(\bmm) = \sum_{\substack{\pp \in \Supp(\bmm)\\ \dim(\BB/\pp)=d}} \text{\normalfont length}_{\BB_\pp}\big(\bmm_\pp\big)\,e_\bn\left(\BB/\pp\right).
		$$
	\end{enumerate}
	\begin{proof}
		$(i)$ It follows from the fact that $\Hilb_{\bmm(-\nu)}(\ttt)=\ttt^\nu\Hilb_\bmm(\ttt)$.
		
		$(ii)$ Since 
		$$\Hilb_\bmm(\ttt)=\frac{(1-t_1)^{d-\dim(\bmm^\prime)}}{(1-t_1)^{d-\dim(\bmm^\prime)}}\Hilb_{\bmm^{\prime}}(\ttt)+\frac{(1-t_1)^{d-\dim(\bmm^{\prime\prime})}}{(1-t_1)^{d-\dim(\bmm^{\prime\prime})}}\Hilb_{\bmm^{\prime\prime}}(\ttt),
		$$
		the result is obtained from \autoref{thm_struc_Hilb_series} and \autoref{lem_unique_values}.
		
		$(iii)$ Take a finite filtration 
		$$
		0 = \bmm_0 \subset \bmm_1 \subset \cdots \subset \bmm_n=\bmm
		$$
		of $\bmm$ such that $\bmm_l/\bmm_{l-1} \cong \left(\BB/\pp_l\right)(-\nu_l)$ where $\pp_l \subset \BB$ is a graded prime ideal with dimension $\dim(\BB/\pp_l)\le d$ and $\nu_l \in \ZZ^r$.
		From parts $(i)$ and $(ii)$, it follows that $e_\bn(\bmm)=\sum_{l=1}^ne_\bn\left(\BB/\pp_l\right)$.
		The inequality $\dim\left(\BB/\pp_l\right)\le d-1$ implies that $e_\bn\left(\BB/\pp_l\right)=0$.
		On the other hand, for any $\pp \in \Supp(\bmm)$ with $\dim\left(\BB/\bmm\right)=d$, the cardinality of the set $\{ l \mid \pp = \pp_l \} \subseteq \{1,2,\ldots,n\}$ is equal to $\text{length}_{\BB_\pp}\big(\bmm_\pp\big)$.
		So, the result follows.
	\end{proof}
\end{lemma}

The following theorem is a generalization of the results of \cite[Theorem 4.3]{HERMANN_MULTIGRAD} and \cite[Theorem 4.1]{VERMA_BIGRAD}.
Additionally, it characterizes when multiplicities of negative type are positive, that is, when $e_\bn(\bmm)>0$ for $\bn \not\ge \mathbf{0}$ and some finitely generated graded $\BB$-module $\bmm$. 

\begin{theorem}
	\label{thm_mult_gr_mixed}
	Assume \autoref{setup_general}.
	Let $\bmm$ be a finitely generated graded $\BB$-module with $\dim(\bmm)=d$.
	Then, the following statements hold:
	\begin{enumerate}[(i)]
		\item We have the equality
		$$
		e\big(\bmm^\gr\big) = \sum_{\substack{\bn \ge \mathbf{-1}\\ \lvert\bn+\mathbf{1}\rvert=d}} e_\bn\left(\bmm\right).
		$$ 
		\item $\lbrace \pp \in \Supp\left(\bmm\right) \mid \dim(\BB/\pp) = d \rbrace \cap V(\nn) \neq \emptyset$ if and only if $
		e_\bn(\bmm) >  0
		$
		for some $\bn \not\ge \mathbf{0}$.
	\end{enumerate}

	\begin{proof}
		$(i)$
		The Hilbert series $\Hilb_{\bmm^\gr}(t)$ of the single-graded module $\bmm^\gr$ can be (uniquely) written as $\Hilb_{\bmm^\gr}(t)=\frac{Q(t)}{(1-t)^d}$ for some $Q(t) \in \ZZ[t,t^{-1}]$ (see, e.g., \cite[\S 4.1]{BRUNS_HERZOG}), and can obtained from $\Hilb_\bmm(\ttt)$ by making the substitutions $t_i \mapsto t$ for $1 \le i \le r$.
		Hence, taking a decomposition of $\Hilb_\bmm(\ttt)$ from \autoref{thm_struc_Hilb_series} gives the equation 
		$$
		\frac{Q(t)}{(1-t)^d} = \frac{\sum_{\substack{\bn \ge \mathbf{-1}\\ \lvert\bn+\mathbf{1}\rvert=d}} Q_{\bn+\mathbf{1}}(t,\ldots,t)}{(1-t)^d}.
		$$
		Therefore, the equality 
		$$
		e\big(\bmm^\gr\big) = \sum_{\substack{\bn \ge \mathbf{-1}\\ \lvert\bn+\mathbf{1}\rvert=d}} e_\bn\left(\bmm\right)
		$$		
		follows from the fact that $e\left(\bmm^\gr\right)=Q(1)$ and $e_\bn\left(\bmm\right)=Q_{\bn+\mathbf{1}}(\mathbf{1})$.

		$(ii)$
		Suppose  that $e_\bm(\bmm)>0$ for some $\mathbf{-1} \le \bm \in \ZZ^r$ such that $\lvert \bm+\mathbf{1} \rvert = d$ and that $m_i=-1$ for some $1 \le i \le r$; and fix such $\bm$ and $i$.
		From \autoref{lem_basic_props}$(iii)$, there exists a prime $\pp \in \Supp(\bmm)$ with $\dim\left(\BB/\pp\right)=d$ such that $e_\bm\left(\BB/\pp\right)>0$.
		By using \autoref{thm_struc_Hilb_series} take a decomposition 
		$$
		\Hilb_{\BB/\pp}(\ttt) = \sum_{\substack{\bn \ge \mathbf{-1}\\\lvert\bn+\mathbf{1}\rvert = d}} \frac{Q_{\bn+\mathbf{1}}(\ttt)}{(\mathbf{1-t})^{\bn+\mathbf{1}}}.
		$$
		Since $e_\bm\left(\BB/\pp\right)>0$, it follows that $Q_{\bm+\mathbf{1}}(\ttt)\neq0$. 
		Let $\ttt^\prime=(t_1,\ldots,t_{i-1},t_{i+1},\ldots,t_r)$ and write $Q_{\bm+\mathbf{1}}(\ttt)=\sum_{j=0}^{c}t_i^jG_j(\ttt^\prime)$ where $G_j(\ttt^\prime) \in \ZZ[\ttt^\prime]$.
		
		For any $F(\ttt) = \sum_{\nu \ge \mathbf{0}} f_\nu t^\nu \in \ZZ[[\ttt]]$ power series we use the notation 
		$$
		{\big\langle F(\ttt)\big\rangle}_c^i := \sum_{\substack{\nu \ge \mathbf{0}\\ \nu_i \le c}} f_\nu t^\nu \in \ZZ[[\ttt]].
		$$
		Clearly, the function ${\big\langle \bullet\big\rangle}_c^i:\ZZ[[\ttt]] \rightarrow \ZZ[[\ttt]]$ is additive.
		
		From the following isomorphism 
		$$
		\BB/\left(\pp,\MM_i^{c+1}\right) \cong \bigoplus_{\substack{\nu \ge \mathbf{0}\\ \nu_i\le c}} \left[\BB/\pp\right]_\nu,
		$$
		we then obtain that 
		$$
		\Hilb_{\BB/\left(\pp,\MM_i^{c+1}\right)}(\ttt)
		={\big\langle\Hilb_{\BB/\pp}(\ttt)\big\rangle}_c^i = \sum_{\substack{\bn \ge \mathbf{-1}\\\lvert\bn+\mathbf{1}\rvert = d}} {\Big\langle\frac{Q_{\bn+\mathbf{1}}(\ttt)}{(\mathbf{1-t})^{\bn+\mathbf{1}}}\Big\rangle}_c^i.
		$$
		Since ${\Big\langle\frac{Q_{\bm+\mathbf{1}}(\ttt)}{(\mathbf{1-t})^{\bm+\mathbf{1}}}\Big\rangle}_c^i=\frac{Q_{\bm+\mathbf{1}}(\ttt)}{(\mathbf{1-t})^{\bm+\mathbf{1}}}$, \autoref{lem_unique_values} and \autoref{thm_struc_Hilb_series} imply that 
		$$
		\dim\left(\BB/\left(\pp,\MM_i^{c+1}\right)\right) = \dim\left(\BB/\pp\right)=d.
		$$
		But then we get the containment $\pp \supseteq \MM_i \supseteq \nn$ because $\pp$ is a prime ideal.
		Therefore, it follows that $\lbrace \pp \in \Supp\left(\bmm\right) \mid \dim(\BB/\pp) = d \rbrace \cap V(\nn) \neq \emptyset$.
		
		For the reverse implication, assume that 
		$$
		Z = \lbrace \pp \in \Supp\left(\bmm\right) \mid \dim(\BB/\pp) = d \rbrace \cap V(\nn) \neq \emptyset,
		$$ 
		and choose $\pp \in Z$.
		Since $\pp \supseteq \nn$, $\left[\BB/\pp\right]_{\nu}=0$ for $\nu > \mathbf{0}$ and so it follows that a decomposition from \autoref{thm_struc_Hilb_series} can be written as
		$$
		\Hilb_{\BB/\pp}(\ttt) = \sum_{\bn \in \Sigma} \frac{Q_{\bn+\mathbf{1}}(\ttt)}{(\mathbf{1-t})^{\bn+\mathbf{1}}}
		$$
		where $\Sigma = \lbrace \bn \in \ZZ^r \mid \bn \ge \mathbf{-1}, \lvert \bn + \mathbf{1}\rvert=d \text{ and } n_i=-1 \text{ for some } 1\le i \le r\rbrace$.
		Therefore, \autoref{lem_basic_props}$(iii)$ implies that $e_\bn(\bmm) > 0$ for some $\mathbf{-1} \le \bn \in \ZZ^r$ such that $\lvert \bn+\mathbf{1} \rvert = d$ and that $n_i=-1$ for some $1 \le i \le r$.
		
		So, the result of the theorem follows.
	\end{proof}
\end{theorem}

\section{Relation with mixed multiplicities via Hilbert polynomials}
\label{sect3}

In this section we relate the mixed multiplicities introduced in \autoref{def_mixed_mult} with the usual mixed multiplicities defined in terms of Hilbert polynomials.
Here we also study how mixed multiplicities behave after taking quotients by filter-regular sequences.
Throughout this section we continue using all the notations and conventions of the previous section.

The following theorem shows that in a multigraded setting we can define a multigraded Hilbert polynomial, which provides the usual approach for defining mixed multiplicities.
Below in \autoref{thm_equality_mults} we obtain a different proof of this result.

\begin{theorem}[{\cite[Theorem 4.1]{HERMANN_MULTIGRAD}}]
	\label{thm_hilbert_polynomial}
	Assume \autoref{setup_general}.
	Let $\bmm$ be a finitely generated graded $\BB$-module.
	Then, there exists a polynomial $P_{\bmm}(\XX)=P_\bmm(X_1,\ldots,X_r) \in \QQ[\XX]=\QQ[X_1,\ldots,X_r]$ which can be written as 
	$$
	P_{\bmm}(\XX) = \sum_{n_1,\ldots,n_r \ge 0} e(n_1,\ldots,n_r)\binom{X_1+n_1}{n_1}\cdots \binom{X_r+n_r}{n_r},
	$$
	where $e(n_1,\ldots,n_r) \in \ZZ$, and that satisfies the following conditions:
	\begin{enumerate}[(i)]
		\item The total degree of $P_\bmm(\XX)$ is equal to $\dim\left(\Supp_{++}(\bmm)\right)$.
		\item $e(n_1,\ldots,n_r) \ge 0$ for any $n_1+\cdots+n_r=\dim\left(\Supp_{++}(\bmm)\right)$.
		\item $P_\bmm(\nu)=\text{\normalfont length}_A\left(\left[\bmm\right]_\nu\right)$ for all $\nu \gg \mathbf{0}$.
	\end{enumerate}
\end{theorem} 

Motivated by the previous theorem, we have the following definition which goes back to the work of van der Waerden (\cite{VAN_DER_WAERDEN}).

\begin{definition}
	\label{def_mixed_mult_poly}
	Assume \autoref{setup_general}. 
	Let $\bmm$ be a finitely generated graded $\BB$-module. 
	Let $d_{++}=\dim\left(\Supp_{++}(\bmm)\right)$. 
	Let $P_\bmm(\XX)=P_\bmm(X_1,\ldots,X_r)$ be the polynomial
	$$
	P_{\bmm}(\XX) = \sum_{n_1,\ldots,n_r \ge 0} e(n_1,\ldots,n_r)\binom{X_1+n_1}{n_1}\cdots \binom{X_r+n_r}{n_r}
	$$
	obtained according to \autoref{thm_hilbert_polynomial}.
	For any $\bn=(n_1,\ldots,n_r) \in \NN^r$ with $\lvert \bn\rvert \ge d_{++}$, the \textit{mixed multiplicity of $\bmm$ of type $\bn$ defined in terms of Hilbert polynomials} is given by 
	$$
	e\left(\bn;\bmm\right) := \begin{cases}
	e(n_1,\ldots,n_r) \quad\text{ if } \lvert \bn\rvert = d_{++}\\
	0 \quad\;\;\,\qquad\qquad\; \text{ if } \lvert \bn\rvert > d_{++}.
	\end{cases} 
	$$	
\end{definition}

The next simple lemma will be needed to relate \autoref{def_mixed_mult} and  \autoref{def_mixed_mult_poly}.  

\begin{lemma}
	\label{lem_properties_quot_torsion}
	Let $\bmm$ be a finitely generated graded $\BB$-module. Then, the following statements hold:
	\begin{enumerate}[(i)]
		\item ${\left[\bmm\right]}_\nu = {\left[\bmm / \HH^0_{\nn}(\bmm)\right]}_\nu$ for all $\nu \gg \mathbf{0}$.
		\item $\Ass\left(\bmm / \HH^0_{\nn}(\bmm)\right)=\Ass(\bmm) \setminus V(\nn)$.
		\item $\dim\left(\Supp_{++}(\bmm)\right) = \dim\left(\bmm / \HH^0_{\nn}(\bmm)\right) - r$.
	\end{enumerate}
	\begin{proof}
		$(i)$ Follows from the fact that $\nn^k \cdot \HH_\nn^0(\bmm)=0$ for some $k \ge 0$.
		
		$(ii)$ It is well-known (see, e.g., \cite[Proposition 3.13]{EISEN_COMM}).
		
		$(iii)$ By using part $(ii)$ and $\Supp_{++}\left(\bmm\right)=\Supp_{++}\left(\bmm/\HH_\nn^0(\bmm)\right)$,  we get
		$$
		\dim\left(\Supp_{++}\left(\bmm\right)\right) = \sup\big\lbrace \dim\big(\multProj\left(\BB/\pp\right)\big) \mid \pp \in \Supp_{++}(\bmm) \big\rbrace
		$$
		and 
		$$
		\dim\left(\bmm/\HH_\nn^0(\bmm)\right) = \sup\big\lbrace \dim\left(\BB/\pp\right) \mid \pp \in \Supp_{++}(\bmm) \big\rbrace.
		$$
		Finally, for any $\pp \not\supseteq \nn$ it is known that $\dim\left(\multProj(\BB/\pp)\right)=\dim(\BB/\pp)-r$ (see, e.g., \cite[Lemma 1.1]{HERMANN_MULTIGRAD}, \cite[Lemma 1.2]{HYRY_MULTIGRAD}).
		If $\Supp_{++}(\bmm)=\emptyset$, then by convention we have $\dim\left(\Supp_{++}(\bmm)\right)=-\infty$ and $\dim\left(\bmm/\HH_{\nn}^0(\bmm)\right)=-\infty$.
	\end{proof}
\end{lemma}

The following theorem shows that the mixed multiplicities in terms of Hilbert polynomials can be expressed as the ones in terms of Hilbert series after modding out the torsion with respect to $\nn$.
Additionally, we provide a different proof of \autoref{thm_hilbert_polynomial}.

\begin{theorem}
	\label{thm_equality_mults}
	Assume \autoref{setup_general}. 
	Let $\bmm$ be a finitely generated graded $\BB$-module. 
	Let $d_{++}=\dim\left(\Supp_{++}(\bmm)\right)$. 
	Then, the following statements hold:
	\begin{enumerate}[(i)]
		\item For $\nu \gg \mathbf{0}$, the function $\text{\normalfont length}_A\left(\left[\bmm\right]_\nu\right)$ becomes a polynomial in $\QQ[\XX]$ of total degree $d_{++}$ which can be written in the form
		$$
		\sum_{\lvert \bn \rvert = d_{++}} \frac{e_\bn\Big(\bmm/\HH_\nn^0(\bmm)\Big)}{\bn!}\XX^{\bn}\, +\, \left({\normalfont} \text{terms of total degree} < d_{++} \right).
		$$
		\item For each $\bn \in \NN^r$ with $\lvert \bn \rvert \ge d_{++}$ we have the equality 
		$$
		e(\bn;\bmm) = e_\bn\Big(\bmm/\HH_\nn^0(\bmm)\Big).
		$$		
	\end{enumerate}
	\begin{proof}
		$(i)$
			For simplicity of notation set $\obmm=\bmm/\HH_\nn^0(\bmm)$.
			Note that \autoref{lem_properties_quot_torsion}$(iii)$ yields $d_{++}=\dim(\obmm)-r$.
			From \autoref{thm_struc_Hilb_series} take a decomposition
			$$
			\Hilb_\obmm(\ttt)=\sum_{\substack{\lvert\bn\rvert = d_{++}\\\bn \ge \mathbf{-1}}} \frac{Q_{\bn+\mathbf{1}}(\ttt)}{(\mathbf{1-t})^{\bn+\mathbf{1}}}.
			$$
			None of the hypotheses or conclusions change if we consider instead the shifted module $\bmm(-\nu)$ for any $\nu \ge \mathbf{0}$; therefore, without any loss of generality we may assume that $Q_{\bn+\mathbf{1}}(\ttt) \in \ZZ[\ttt]$ are polynomials in $\ttt$.
			Let 
			$$
			F(\ttt)=\Hilb_\obmm(\ttt)-\sum_{\substack{\lvert\bn\rvert = d_{++}\\\bn \ge \mathbf{0}}} \frac{Q_{\bn+\mathbf{1}}(\mathbf{1})}{(\mathbf{1-t})^{\bn+\mathbf{1}}} =\sum_{\substack{\lvert\bn\rvert = d_{++}\\\bn \ge \mathbf{-1}}} \frac{P_{\bn+\mathbf{1}}(\ttt)}{(\mathbf{1-t})^{\bn+\mathbf{1}}}
			$$ 
			where 
			$$
			P_{\bn+\mathbf{1}}(\ttt) = \begin{cases}
				Q_{\bn+\mathbf{1}}(\ttt)-Q_{\bn+\mathbf{1}}(\mathbf{1})  \quad \text{ if } \bn \ge \mathbf{0}\\
				Q_{\bn+\mathbf{1}}(\ttt) \qquad\qquad\quad\quad\text{ if } \bn \not\ge \mathbf{0}.
			\end{cases}
			$$
			
			From \autoref{lem_properties_quot_torsion}$(ii)$ and \autoref{thm_mult_gr_mixed}$(ii)$ we obtain that $e_\bn(\obmm) > 0$ only when $\bn \ge \mathbf{0}$.
			Then, \autoref{thm_struc_Hilb_series}$(ii)$ implies that, for any $\bn \not\ge \mathbf{0}$,  $P_{\bn+\mathbf{1}}(\ttt)=Q_{\bn+\mathbf{1}}(\ttt)$ and $(\mathbf{1-t})^{\bn+\mathbf{1}}$ are both divisible by some $(1-t_i)$.
			For $\bn \ge \mathbf{0}$, since $P_{\bn+\mathbf{1}}(\mathbf{1})=0$, we can write $P_{\bn+\mathbf{1}}(\ttt)=\sum_{\lvert \alpha \rvert > 0}p_{\bn+\mathbf{1}}^\alpha (\mathbf{1-t})^\alpha$ where $p_{\bn+\mathbf{1}}^\alpha \in \ZZ$.
			By summing up, we conclude that $F(\ttt)$ can be written in the following form 
			$$
			F(\ttt) = \sum_{\substack{\lvert\bn\rvert < d_{++}\\\bn \ge \mathbf{-1}}} \frac{G_{\bn+\mathbf{1}}(\ttt)}{(\mathbf{1-t})^{\bn+\mathbf{1}}}
			$$
			where $G_{\bn+\mathbf{1}}(\ttt) \in \ZZ[\ttt]$, i.e., $F(\ttt)$ can be written as the sum of rational functions with quotients of total order strictly less than $r+d_{++}$.
			
			Expanding these series we get that 
			$$
			\Hilb_{\obmm}(\ttt) = \sum_{\nu}\left(\sum_{\substack{\lvert\bn\rvert = d_{++}\\\bn \ge \mathbf{0}}} Q_{\bn+\mathbf{1}}(\mathbf{1})\binom{\nu+\bn}{\bn}
			\right)\ttt^{\nu}
			+ \sum_{\nu}\left(\sum_{\substack{\lvert\bn\rvert < d_{++}\\\bn \ge \mathbf{-1}}} G_{\bn+\mathbf{1}}(\ttt)\binom{\nu+\bn}{\bn}
			\right)\ttt^\nu 
			$$
			where we use the notation
			$$
			\binom{\nu+\bn}{\bn}=\binom{\nu_1+n_1}{n_1}\cdots\binom{\nu_r+n_r}{n_r}.
			$$
			By using \autoref{lem_properties_quot_torsion}$(i)$  and expanding the polynomials $G_{\bn+\mathbf{1}}(\ttt)$, we obtain that for $\nu\gg \mathbf{0}$ the function $\text{length}_A(\left[\bmm\right]_\nu)$ becomes a polynomial in $\QQ[\XX]$ of the form
			$$
			\sum_{\lvert \bn \rvert = d_{++}} \frac{Q_{\bn+\mathbf{1}}(\mathbf{1})}{\bn!}\XX^{\bn}\, +\, \left({\normalfont} \text{terms of total degree} < d_{++} \right).
			$$
			Therefore, we are done since $e_\bn(\obmm)=Q_{\bn+\mathbf{1}}(\mathbf{1})$. 
		
			$(ii)$ It follows from part $(i)$.
	\end{proof}
\end{theorem}

To further study mixed multiplicities we use the concept of \textit{filter-regular sequence}, also known as \textit{almost regular sequence} (see, e.g., \cite{TRUNG_POSITIVE,TRUNG_REDUCTIONEXPONENT}, \cite[\S 4.3.1]{HERZOG_HIBI_MONOMIALS}).

\begin{definition}
	\label{def_filter_reg_elems}
	Let $1 \le i \le r$ and $\bmm$ be a finitely generated graded $\BB$-module.
	A homogeneous element is said to be \textit{filter-regular on $\bmm$} if $z \not\in \pp$ for all associated primes $\pp \in \Ass(\bmm)$ of $\bmm$ such that $\pp \not\supseteq \nn$.
	A sequence of homogeneous elements $z_1,\ldots,z_m \in \BB$ is said to be \textit{filter-regular on $\bmm$} if $z_j$ is a filter-regular element on $\bmm/\left(z_1,\ldots,z_{j-1}\right)\bmm$ for all $1 \le j \le m$.
\end{definition}

\begin{lemma}
	\label{lem_filter_regular}
	Let $1 \le i \le r$,  $z \in \left[\BB\right]_{\ee_i}$ and $\bmm$ be a finitely generated graded $\BB$-module.
	Then, the following statements are equivalent:
	\begin{enumerate}[(i)]
		\item $\big[\left(0:_{\bmm} z \right)\big]_\nu=0$ for $\nu \in \ZZ^r$ with $\nu \gg \mathbf{0}$.
		\item $\nn^k \cdot  \left(0:_{\bmm} z \right) = 0$ for some $k > 0$.
		\item $\Supp\left(\left(0:_{\bmm} z \right)\right) \subseteq V(\nn)$.
		\item $z$ is filter-regular on $\bmm$.		
	\end{enumerate}	
	\begin{proof}
		The equivalence $(i) \Leftrightarrow (ii)$ is clear.
		For the equivalence $(ii) \Leftrightarrow (iii)$, note that $\Supp(\left(0:_{\bmm} z \right))=V(\Ann(\left(0:_{\bmm} z \right))) \subseteq V(\nn)$ is equivalent to $\nn \subseteq \sqrt{\Ann\left(\left(0:_{\bmm} z \right)\right)}$.
		
		$(iii) \Rightarrow (iv)$ 
		Suppose that $\Supp\left(\left(0:_{\bmm} z \right)\right) \subseteq V(\nn)$ and let $\pp \in \Ass(\bmm)$ such that $\pp \not\supseteq \nn$.
		Then, there is an injection $R/\pp \hookrightarrow \bmm$, and if $z \in \pp$ then we would get the contradiction $0 \neq R_\pp/\pp R_\pp = \left(0:_{R_\pp/\pp R_\pp} zR_\pp\right) \hookrightarrow  \left(0:_{\bmm_\pp} zR_\pp\right) = \left(0:_\bmm z\right)_\pp$.
		Therefore $z \not\in \pp$. 
		
		$(iv) \Rightarrow (iii)$
		Suppose $z$ is filter-regular on $\bmm$.
		Since the minimal primes of $\Ass\left(\left(0:_\bmm z\right)\right)$ and $\Supp\left(\left(0:_\bmm z\right)\right)$ coincide, it is enough to show that $\Ass\left(\left(0:_\bmm z\right)\right) \subseteq V(\nn)$.
		Take $\pp \in \Ass(\bmm)$ such that $\pp \not\supseteq \nn$, then we obtain $\left(0:_\bmm z\right)_\pp = \left(0:_{\bmm_\pp} zR_\pp\right)=\left(0:_{\bmm_\pp} R_\pp\right)=0$ because $z \not\in \pp$.
		Therefore $\Ass\left(\left(0:_\bmm z\right)\right) \subseteq V(\nn)$.
	\end{proof}
\end{lemma}

The following lemma shows that under the assumption that the residue field $\kk=A/\nnn$ is infinite we can always find filter-regular elements.  

\begin{lemma}
	\label{lem_general_filter_reg_seq}
	Let $1 \le i \le r$, $\bmm$ be a finitely generated graded $\BB$-module and suppose that the residue field $\kk=A/\nnn$ of $A$ is infinite.
	Let $V$ be the finite dimensional $\kk$-vector space $V={\left[\BB\right]}_{\ee_i} \otimes_A \kk$, choose a basis for $V$ and consider the Zariski topology on $V\cong \kk^{\dim_\kk(V)}$.
	Denote by $\pi$ the canonical map $\pi: {\left[\BB\right]}_{\ee_i} \rightarrow V={\left[\BB\right]}_{\ee_i} \otimes_A \kk$.
	Then 
	$$
	U = \lbrace w \in V \mid w = \pi(z) \text{ for some } z\in \left[\BB\right]_{\ee_i} \text{ filter-regular on } \bmm \rbrace 
	$$
	is a dense open subset of $V$.
	\begin{proof}
		Let $\Ass(\bmm) \setminus V(\nn) = \{\pp_1,\ldots,\pp_m\}$.
		Since $\pp_j \not\supseteq \MM_i$, then Nakayama's lemma implies that $V_j = \pi\left(\pp_j \cap {\left[\BB\right]}_{\ee_i}\right)$ is a proper linear subspace of $V$.
		Since $\nnn$ is nilpotent, $\pp_j \supset \nnn$ and in particular $\pp_j \supset \nnn{\left[\BB\right]}_{\ee_i}$. 
		Then, for any $z \in \left[\BB\right]_{\ee_i}$ we have that $z$ is filter-regular on $\bmm$ if and only if $\pi(z) \not\in V_1 \cup \cdots \cup V_m$.
		Therefore, using that $\kk$ is an infinite field, the result follows  and we obtain $U = V \setminus \left(V_1 \cup \cdots \cup V_m\right)$.
	\end{proof}
\end{lemma}

\begin{remark}
	As customary, the assumption on the infiniteness of $\kk=A/\nnn$ can be achieved by making the faithfully flat base change $\BB \otimes_A B$ where $B = {A[x]}_{\nnn A[x]}$ and $x$ is an indeterminate.
	Note that the residue field of $B$ is $\kk(x)$ and that $\text{length}_A\left(\left[\bmm\right]_\nu\right) =\text{length}_{B} \left(\left[\bmm \otimes_A B\right]_\nu\right)$ for all $\nu \in \ZZ^r$ and $\bmm$ a finitely generated graded $\BB$-module.
\end{remark}

The next proposition shows that mixed multiplicities behave nicely when taking quotient by a filter-regular element. 

\begin{lemma}
	\label{lem_reduce_filter_reg}
	Let $1\le i \le r$, $\bmm$ be a finitely generated graded $\BB$-module with $d_{++}=\dim\left(\Supp_{++}(\bmm)\right)$ and $z \in \left[\BB\right]_{\ee_i}$ be a filter-regular element on $\bmm$.
	Then, we obtain that $\dim\left(\Supp_{++}(\bmm/z \bmm)\right) \le d_{++} - 1$ and 
	$$
	e(\bn;\bmm) = e\left(\bn-\ee_i;\bmm/z\bmm\right)
	$$
	for all $\ee_i \le \bn \in \NN^r$ such that $\lvert \bn \rvert \ge d_{++}$.
	\begin{proof}
		The four-term exact sequence 
		$$
		0 \rightarrow \left(0:_{\bmm} z\right)(-\ee_i) \rightarrow \bmm(-\ee_i) \xrightarrow{z} \bmm \rightarrow \bmm/z\bmm \rightarrow 0,
		$$
		\autoref{thm_hilbert_polynomial} (or \autoref{thm_equality_mults}) and \autoref{lem_filter_regular} give the equation 
		$$
		P_{\bmm/z\bmm}(\XX) = P_{\bmm}(\XX) - P_{\bmm}(\XX-\ee_i).
		$$
		It is easy to check that $P_{\bmm}(\XX) - P_{\bmm}(\XX-\ee_i)$ is a polynomial of total degree smaller or equal than $d_{++}-1$ and that its coefficient of degree $\bn-\ee_i$ is equal to $\frac{e(\bn;\bmm)}{(\bn-\ee_i)!}$.
		So, the result follows from \autoref{thm_hilbert_polynomial} (or \autoref{thm_equality_mults}).
	\end{proof}
\end{lemma}

Finally, we now extend \cite[Theorem 2.4]{TRUNG_POSITIVE} to a multigraded setting.
It reduces the computation of mixed multiplicities to the computation of a single-graded module.

\begin{theorem}
	\label{thm_mixed_mult_as_single_grad}
	Assume \autoref{setup_general}.
	Let $\bmm$ be a finitely generated graded $\BB$-module. 
	Let $d_{++}=\dim\left(\Supp_{++}(\bmm)\right)$ and $\bn \in \NN^r$ with $\lvert \bn\rvert=d_{++}$.
	For $1 \le i \le r$, let 
	$
	\mathbf{z}_i=z_{i,1}\ldots,z_{i,n_i} \in \left[\BB\right]_{\ee_i}.
	$
	Suppose that $\mathbf{z}_1,\ldots,\mathbf{z}_r$ is a filter-regular sequence on $\bmm$.
	Set $\EE$ to be the single-graded module
	$$
	\EE = {\left(\frac{\bmm/(\mathbf{z}_1,\ldots,\mathbf{z}_r)\bmm}{\HH_\nn^0\Big(\bmm/(\mathbf{z}_1,\ldots,\mathbf{z}_r)\bmm\Big)}\right)}^{\gr}.
 	$$
 	Then, the following equation holds
 	$$
 	e(\bn;\bmm) = \begin{cases}
		e(\EE) \quad \text{ if } \dim(\EE)=r\\
		0 \qquad\;\,\,\text{otherwise}.
 	\end{cases}
 	$$
 	\begin{proof}
		Applying \autoref{lem_reduce_filter_reg} successively it follows that $e(\bn;\bmm)=e(\mathbf{0};\bmm/(\mathbf{z}_1,\ldots,\mathbf{z}_r)\bmm)$.
		Therefore, from \autoref{thm_equality_mults}(ii) and \autoref{thm_mult_gr_mixed} we obtain the result. 
 	\end{proof}
\end{theorem}

\section{Multidegrees of multiprojective schemes}
\label{sect4}

In this section, we study the degrees of multiprojective schemes via the use of mixed multiplicities. 
The main objective here is to obtain a direct generalization of van der Waerden's result \cite{VAN_DER_WAERDEN} in a multigraded setting.
The results exposed in this short section are probably well-known and part of the folklore, but, for the sake of completeness, we include a very short account that depends directly on the previous sections.
Perhaps worthy of mentioning, our approach here is completely based upon the use of filter-regular elements (as introduced in \autoref{def_filter_reg_elems}).
The following setup is used throughout this section. 

\begin{setup}
	\label{setup_mult_degrees}
	Let $\kk$ be a field, $\AAA$ be the standard multigraded polynomial ring 
	$$
	\AAA = \kk[x_{1,0},\ldots,x_{1,d_1}] \otimes_\kk \cdots \otimes_\kk \kk[x_{r,0},\ldots,x_{r,d_r}]
	$$
	and $\PP$ be the corresponding multiprojective space
	$
	\PP = \multProj(\AAA) = \PP_\kk^{d_1} \times_\kk  \cdots \times_\kk \PP_\kk^{d_r}.
	$
\end{setup}

First, we recall the notion of degree for zero-dimensional schemes over $\kk$ (see, e.g., \cite[\S II.3.2]{EISENBUD_HARRIS_SCHEMES}).

\begin{definition}
	Let $\bYY$ be a $\kk$-scheme of finite type  with $\dim(\bYY)=0$.
	The \textit{degree of $\bYY$ relative to $\kk$} is given by 
	$$
	\deg_\kk\left(\bYY\right) := \sum_{y \in \bYY} \left[k(y):\kk\right]\,\text{length}\left(\OO_{\bYY,y}\right),
	$$ 
	where $k(y)$ denotes the residue field of the local ring $\OO_{\bYY,y}$.
\end{definition}

Next we define the multidegrees of closed subschemes of $\PP$ in terms of mixed multiplicities.

\begin{definition}
	\label{def_multdegree}
	Let $\bXX \subset \PP$ be a closed subscheme of $\PP$ defined as $\bXX=\multProj\left(\AAA/\JJ\right)$ where $\JJ \subset \AAA$ is a graded ideal.	
	Let $\bn \in \NN^r$ with $\lvert \bn \rvert \ge \dim(\bXX)$.
	The \textit{multidegree of $\bXX$ of type $\bn$ with respect to $\PP$} is given by 
	$$
	\deg_\PP^\bn\left(\bXX\right) := e\left(\bn;\frac{\AAA}{\JJ}\right).
	$$
\end{definition}

Note that, equivalently, the multidegrees of a closed subscheme of $\PP$ can be defined easily in terms of Chow rings.

\begin{remark}
	\label{rem_chow_ring}
	The Chow ring of $\PP = \PP_\kk^{d_1} \times_\kk \cdots \times_\kk \PP_\kk^{d_r}$ is given by $$
	A^*(\PP) = \frac{\ZZ[\xi_1,\ldots,\xi_r]}{\left(\xi_1^{d_1+1},\ldots,\xi_r^{d_r+1}\right)}
	$$
	where $\xi_i$ represents the class of the inverse image of a hyperplane of $\PP_\kk^{d_i}$ under the canonical projection $\pi_i: \PP \rightarrow \PP_\kk^{d_i}$. 
	If $\bXX \subset \PP$ is a closed subscheme of $\PP$ of dimension $d=\dim(\bXX)$, then the class of the cycle associated to $\bXX$ coincides with
	$$
	\left[\bXX\right] = \sum_{\substack{0 \le n_i \le d_i\\ \lvert \bn \rvert=d}} \deg_\PP^\bn\left(\bXX\right)\, \xi_1^{d_1-n_1}\cdots\xi_r^{d_r-n_r} \;\in A^*(\PP).
	$$
\end{remark}

\begin{definition}
For a closed subscheme $\bXX = \multProj\left(\AAA/\JJ\right) \subset \PP$, we say that $H \subset \PP$ is a \textit{filter-regular hyperplane on $\bXX$} if $H$ is given by $H=V_{++}(h)$ where $h \in {\left[\AAA\right]}_{\ee_i}$, for some $1 \le i \le r$, is a filter-regular element on $\AAA/\JJ$.
Similarly, we say that $H_1,\ldots,H_m \subset \PP$ is a \textit{filter-regular sequence of hyperplanes on $\bXX$} if $H_k$ is a filter-regular hyperplane on the closed subscheme $\bXX \cap H_1 \cap \cdots \cap H_{k-1}$ for all $1 \le k \le m$.
\end{definition}

We say that $H \subset \PP$ is a hyperplane in the $i$-th component of $\PP$ if $H=V_{++}(h)$ for some $h \in {\left[\AAA\right]}_{\ee_i}$. 

\begin{remark}
	(i)  A property $P$ is said to be satisfied by a general hyperplane in the $i$-th component of $\PP$, if there exists a dense open subset $U$ of ${\left[\AAA\right]}_{\ee_i}$ with the Zariski topology such that every hyperplane in $U$ satisfies the property $P$.
	\smallskip
	
	\noindent
	(ii) If we fix a closed subscheme $\bXX \subset \PP$ and assume that $\kk$ is an infinite field, then \autoref{lem_general_filter_reg_seq} implies that a sequence $H_1,\ldots,H_m \subset \PP$ of general hyperplanes will be a filter-regular sequence of hyperplanes on $\bXX$.
\end{remark}

Since it could be of interest, we do not assume that the field $\kk$ is infinite and we express the following result in terms of filter-regular hyperplanes.

\begin{theorem}
	\label{thm_van_der_Waerden}
	Assume \autoref{setup_mult_degrees}.
	Let $\bXX \subset \PP$ be a closed subscheme of $\PP$.
	Let $\bn \in \NN^r$ with $\lvert \bn \rvert =\dim(\bXX)$.
	For $1 \le i \le r$, let $H_{i,1},\ldots,H_{i,n_i} \subset \PP$ be a sequence of  hyperplanes in the $i$-th component of $\PP$.
	Suppose that 
	$$
	H_{1,1},\ldots,H_{1,n_1},\;\ldots,\; H_{i,1},\ldots,H_{i,n_i}, \;\ldots,\; H_{r,1},\ldots,H_{r,n_r} \subset \PP
	$$ 
	is a filter-regular sequence of hyperplanes on $\bXX$.
	Then, the following equality holds 
	$$
	\deg_\PP^\bn(\bXX) = \deg_\kk\left(\bXX \,\cap\, \left(\bigcap_{\substack{1 \le i \le r\\1 \le j \le n_i}}H_{i,j}\right)\right).
	$$
	\begin{proof}
		Suppose that $\bXX=\multProj(\BB)$ with $\BB=\AAA/\JJ$.
		Let $h_{i,j} \in {\left[\AAA\right]}_{\ee_i}$ such that $H_{i,j}=V_{++}(h_{i,j}) \subset \PP$.
		The closed subscheme 
		$$
		\bYY = \bXX \,\cap\, \left(\bigcap_{\substack{1 \le i \le r\\1 \le j \le n_i}}H_{i,j}\right)
		$$
		can be expressed as $\bYY=\multProj\big(\BB/\left(h_{1,1},\ldots,h_{r,n_r}\right)\BB\big)$.
		
		By using \autoref{lem_reduce_filter_reg} and \autoref{thm_hilbert_polynomial} (or \autoref{thm_equality_mults}), it follows that either $\dim(\bYY)=0$ or $\bYY=\emptyset$ and that 
		$$\deg_\PP^\bn(\bXX)=e(\bn;\BB)=e\big(\mathbf{0}; \BB/\left(h_{1,1},\ldots,h_{r,n_r}\right)\BB\big).
		$$
		From Serre's Vanishing Theorem (see \cite[Lemma 4.2]{Kleiman_geom_mult} for a multigraded setting) we get that 
		$$
		e\big(\mathbf{0}, \BB/\left(h_{1,1},\ldots,h_{r,n_r}\right)\BB\big) = \dim_\kk\left({\left[\BB/\left(h_{1,1},\ldots,h_{r,n_r}\right)\BB\right]}_\bm\right) = \dim_\kk\left(\HH^0\left(\bYY, \OO_{\bYY}(\bm)\right)\right)
		$$
		for $\bm \gg \mathbf{0}$.
		Since $\dim(\bYY)=0$, we obtain that $\bYY \cong \coprod_{y \in \bYY} \Spec(\OO_{\bYY,y})$ and that $\HH^0\left(\bYY, \OO_{\bYY}(\bm)\right)=\HH^0\left(\bYY, \OO_{\bYY}\right)$ for any $\bm$ (see, e.g., \cite[Proposition 5.11]{GORTZ_WEDHORN}).
		By summing up, we have
		$$
		\deg_\PP^\bn(\bXX) = \dim_\kk\left(\HH^0\left(\bYY, \OO_{\bYY}\right)\right) = \sum_{y \in \bYY} \left[k(y):\kk\right]\,\text{length}\left(\OO_{\bYY,y}\right).
		$$
		So, the result follows.
	\end{proof}
\end{theorem}

\section{Projective degrees of rational maps}
\label{sect5}

During this section, we concentrate on the projective degrees of a rational map.
These numbers are defined as the multidegrees of the graph of a rational map.
As applications, we provide explicit formulas for the projective degrees of rational maps determined by perfect ideals of height two or by Gorenstein ideals of height three.
The following setup will be used throughout this section.

\begin{setup}
	\label{setup_projective_degrees}
	Let $\kk$ be a field, $R$ be the polynomial ring $R=\kk[x_0,\ldots,x_d]$ and $\mm \subset R$ be the graded irrelevant ideal $\mm=\left(x_0,\ldots,x_d\right)$.
	Let $n \ge d$ and $\FF : \PP_\kk^d = \Proj(R) \dashrightarrow  \PP_\kk^n$ be a rational map defined by $n+1$ homogeneous elements $\mathbf{f}=\{f_0,  \ldots, f_n\} \subset R$ of the same degree $\delta > 0$.
	Let $I \subset R$ be the homogeneous ideal $I=\left(f_0,\ldots,f_n\right)$ and $\PP$ be the biprojective space $\PP=\PP_\kk^d \times_\kk \PP_\kk^n$.
	Let $Y \subseteq \PP_\kk^n$ and $\Gamma \subseteq \PP$ be the closed subschemes given as the closures of the image and the graph of $\FF$, respectively.
\end{setup}

\begin{definition}
	\label{def_proj_degrees}
	For $0 \le i \le d$, the $i$-th \textit{projective degree} of $\FF:\PP_\kk^d \dashrightarrow \PP_\kk^n$ is given by 
	$$
	d_i(\FF) := \deg_\PP^{i,d-i}\left(\Gamma\right)
	$$
	(see \autoref{def_multdegree}).
\end{definition}

Equivalently, the projective degrees of a rational map can be defined as in \autoref{rem_chow_ring} (see, e.g., \cite[Example 19.4]{HARRIS} and \cite[\S 7.1.3]{DOLGACHEV}).  

First, we describe $Y$ and $\Gamma$ in algebraic terms as follows (for more details, see, e.g., \cite[\S 3]{SPECIALIZATION_RAT_MAPS}).
Let $\AAA$ be the bigraded polynomial ring $\AAA=R[y_0,\ldots,y_n]$ where $\bideg(x_i)=(1,0)$ and $\bideg(y_i)=(0,1)$.
The Rees algebra $\Rees(I)=\bigoplus_{q=0}^\infty I^qt^q \subset R[t]$ gives the bihomogeneous coordinate ring of $\Gamma$ and can be presented as a quotient of $\AAA$ via the canonical $R$-epimorphism
\begin{eqnarray*}
	\label{presentation_Rees}
	\Psi: \AAA & \twoheadrightarrow & \Rees(I) \subset R[t] \\ \nonumber
	y_i & \mapsto & f_it.
\end{eqnarray*}
The standard graded $\kk$-algebra $S = \kk[f_0,\ldots,f_n]=\bigoplus_{q=0}^\infty\left[I^q\right]_{q\delta}$ gives the homogeneous coordinate ring of $Y$ and we have the canonical epimorphism $\kk[y_0,\ldots,y_n] \twoheadrightarrow S, \; y_i \mapsto f_i$.
In geometrical terms, we obtain the closed immersions $\Gamma = \biProj(\Rees(I)) \hookrightarrow \PP = \biProj(\AAA)$ and $Y = \Proj(S) \hookrightarrow \PP_\kk^n=\Proj(\kk[y_0,\ldots,y_n])$.

Our main tool for the computation of the projective degrees of a rational map will be the saturated special fiber ring.

\begin{definition}[\cite{MULTPROJ}]
	The \textit{saturated special fiber ring} of $I$ is given by the graded $\kk$-algebra
	$$
	\sF(I) := \bigoplus_{q=0}^\infty {\left[\big(I^q:\mm^\infty\big)\right]}_{q\delta}.
	$$
\end{definition} 

A very important feature of $\sF(I)$ is that it is a finitely generated $S$-module and that its multiplicity is equal to $e\left(\sF(I)\right)=\deg(\FF)\,\deg_{\PP_\kk^n}(Y)$ (i.e., the product of the degrees of the map $\FF$ and its image $Y$; see \cite[Theorem 2.4]{MULTPROJ}).
Although it is well-known that $d_0(\FF)=\deg(\FF)\,\deg_{\PP_\kk^n}(Y)$, in the following theorem we provide a direct proof of the equality $d_0(\FF)=e\left(\sF(I)\right)$.

\begin{theorem}
	\label{thm_equal_d_0_sat_fib}
	Assume \autoref{setup_projective_degrees}.
	If $\FF : \PP_\kk^d  \dashrightarrow  \PP_\kk^n$ is a generically finite map, then 
	$$
	d_0(\FF)=e\left(\sF(I)\right).
	$$
	\begin{proof}
		For notational purposes set $\mathfrak{b}=(y_0,\ldots,y_n)$, $\nn=\mm \cap \mathfrak{b}$ and $\MM=\mm+\mathfrak{b}$.
		The Mayer-Vietoris sequence (see, e.g., \cite[Theorem 3.2.3]{Brodmann_Sharp_local_cohom}) yields the exact sequence 
		$$
		\HH_\MM^i(\Rees(I)) \,\rightarrow\, \HH_\mm^i(\Rees(I)) \oplus \HH_\mathfrak{b}^i(\Rees(I)) \,\rightarrow\, \HH_\nn^i(\Rees(I)) \,\rightarrow\, \HH_\MM^{i+1}(\Rees(I))
 		$$
 		for all $i \ge 0$.
 		Since ${\left[\HH_\MM^i(\Rees(I))\right]}_{(0,j)} = 0$ and ${\left[\HH_\mathfrak{b}^i(\Rees(I))\right]}_{(0,j)} = 0$ for all $j \gg 0$, 
 		it follows that 
 		\begin{equation}
 			\label{eq_equality_cohomologies}
			{\left[\HH_\mm^i(\Rees(I))\right]}_{(0,j)} \cong {\left[\HH_\nn^i(\Rees(I))\right]}_{(0,j)}
 		\end{equation}
		for all $i \ge 0$ and $j \gg 0$.
		
		Let $X=\Proj_{R\text{-gr}}\left(\Rees(I)\right)$ be the projective scheme obtained by considering $\Rees(I)$ as single-graded with the grading of $R$ (i.e., by setting $\deg(x_i)=1$ and $\deg(y_i)=0$). 
		Then, $\sF(I)$ is also given by
		$$
		\sF(I)\cong\HH^0(X,\OO_X)
		$$
		(see \cite[Lemma 2.8]{MULTPROJ}).
		
		We have following relations between sheaf and local cohomologies (see, e.g.,~\cite[Corollary 1.5]{HYRY_MULTIGRAD}, \cite[Appendix A4.1]{EISEN_COMM})
		\begin{equation}
			\label{eq_exact_seq_Gamma}
			0 \rightarrow {\left[\HH_{\nn}^0(\Rees(I))\right]}_{(0,j)} \rightarrow {\left[\Rees(I)\right]}_{(0,j)} \rightarrow  \HH^0\big(\Gamma, \OO_\Gamma(0,j)\big) \rightarrow {\left[\HH_{\nn}^1(\Rees(I))\right]}_{(0,j)} \rightarrow 0
		\end{equation}
		and 
		\begin{equation}
			\label{eq_exact_seq_SatFib}
			0 \rightarrow {\left[\HH_{\mm}^0(\Rees(I))\right]}_{(0,j)} \rightarrow {\left[\Rees(I)\right]}_{(0,j)} \rightarrow  {\left[\HH^0\big(X, \OO_X\big)\right]}_j \rightarrow {\left[\HH_{\mm}^1(\Rees(I))\right]}_{(0,j)} \rightarrow 0.
		\end{equation}		
		Combining \autoref{eq_equality_cohomologies}, \autoref{eq_exact_seq_Gamma} and \autoref{eq_exact_seq_SatFib} we obtain that 
		\begin{equation}
			\label{eq_Sat_Fib_equal_Sheaf_Coh}
			{\left[\sF(I)\right]}_j \,\cong\, {\left[\HH^0(X, \OO_X)\right]}_j \,\cong\,  \HH^0\big(\Gamma, \OO_\Gamma(0,j)\big)
		\end{equation}
		for $j \gg 0$.

		The bigraded Hilbert polynomial of $\Rees(I)$ is given by 
		$$
		P_{\Rees(I)}(u,v)=\sum_{i=0}^d \frac{d_i(\FF)}{i!(d-i)!}\,u^iv^{d-i}\, +\, \left({\normalfont} \text{terms of total degree} < d \right).
		$$
		By using the bigraded version of the Grothendieck-Serre formula (see, e.g.,  \cite[Lemma 4.3]{Kleiman_geom_mult},\cite[Theorem 2.4]{JAYANATHAN_VERMA}), we obtain that 
		$$
		P_{\Rees(I)}(0,j) = \sum_{i\ge0} {(-1)}^i \dim_\kk\Big(\HH^i\left(\Gamma,\OO_\Gamma(0,j)\right)\Big)
		$$
		for all $j$.
		Then, \autoref{eq_equality_cohomologies}, \cite[Corollary 1.5]{HYRY_MULTIGRAD} and \autoref{eq_Sat_Fib_equal_Sheaf_Coh} imply that 
		$$
		P_{\Rees(I)}(0,j) = \dim_\kk\left({\left[\sF(I)\right]}_j\right) \,+\, \sum_{i\ge1} {(-1)}^i \dim_\kk\Big({\left[\HH_\mm^{i+1}(\Rees(I))\right]}_{(0,j)}\Big)
		$$
		for $j \gg 0$.
		From \cite[Corollary 4.5]{SPECIALIZATION_RAT_MAPS} (also, see \cite[Proposition 3.1]{MULTPROJ}), we have that ${\left[\HH_\mm^i(\Rees(I))\right]}_{(0,*)}$ is a finitely generated graded $S$-module with 
		$$
		\dim\left({\left[\HH_\mm^i(\Rees(I))\right]}_{(0,*)}\right)\le d+1-i.
		$$
		Since $S \hookrightarrow \sF(I)$ is an integral extension and $\FF$ is generically finite, it follows that $\dim\left(\sF(I)\right)=\dim(S)=d+1$. 
		Therefore, for $j \gg 0$, $\dim_\kk\left({\left[\sF(I)\right]}_j\right)$ becomes a polynomial of degree $d$ whose leading coefficient coincides with the leading coefficient of $P_{\Rees(I)}(0,v)$.
		This implies that 
		$$
		d_0(\FF) = e\left(\sF(I)\right),
		$$ 
		and so we are done.
	\end{proof}
\end{theorem}

We recall a condition that will assumed in the next subsection.
We say that $I$ satisfies the condition $G_{d+1}$ when
$$
\mu(I_{\pp}) \le \dim(R_{\pp}) \quad \text{ for all }\; \pp \in \Spec(R) \;\text{ such that }\; \HT(\pp)<d+1.
$$				

\begin{remark}
	\label{rem_Fitting_conditions}
	In terms of Fitting ideals, $I$ satisfies the condition $G_{d+1}$ if and only if $\HT(\Fitt_i(I)) > i$ for all $1 \le i < d+1$.
	\begin{proof}
		It follows from \cite[Proposition 20.6]{EISEN_COMM}.
	\end{proof}
\end{remark}

The proposition below contains some reductions to be used later.

\begin{proposition}
	\label{lem_find_filt_reg_rat_map}
	Assume \autoref{setup_projective_degrees} and suppose that $\kk$ is an infinite field. 
	There exist elements $h_1,\ldots,h_d \in {\left[R\right]}_1$ such that, if we set $S_i=R/(h_1,\ldots,h_i)R$ and $J_i=IS_i$ for $1 \le i \le d$, then the following statements hold:
	\begin{enumerate}[(i)]
		\item $d_i(\FF) = e\big(0,d-i;\Rees_{S_i}(J_i)\big)$. 
		\item If $\HT(I)=c$, then $d_i(\FF)=\delta^{d-i}$ for all $d-c+1\le i \le d$.
		\item If $R/I$ is Cohen-Macaulay with minimal graded free resolution 
		$$
		F_\bullet: \quad 0 \rightarrow F_c \rightarrow \cdots \rightarrow F_1 \rightarrow F_0 \rightarrow R/I \rightarrow 0,
		$$
		then, for all $1 \le j \le d-c$, $S_j/J_j$ is Cohen-Macaulay with minimal graded free resolution
		$$
		F_\bullet \otimes_R S_j: \quad 0 \rightarrow F_c\otimes_R S_j \rightarrow \cdots \rightarrow F_1\otimes_R S_j \rightarrow F_0\otimes_R S_j \rightarrow S_j/J_j \rightarrow 0.
		$$
		Additionally, if $I$ satisfies the condition $G_{d+1}$, then $J_j$ satisfies the condition $G_{d+1-j}$ for all $1 \le j \le d-c$.
	\end{enumerate} 
	\begin{proof}
		Set $L_i \subset R$ to be the ideal $L_i=\Fitt_i(I)$ for $1 \le i < d+1$.		
		By using \autoref{lem_general_filter_reg_seq} we can find a sequence $h_1,\ldots,h_d \in {\left[R\right]}_1={\left[\Rees(I)\right]}_{(1,0)}$  which is filter-regular on $\Rees(I)$, on 
		$$
		\gr_I(R) \,=\, \Rees(I) \otimes_R (R/I)  \,=\, \bigoplus_{q=0}^\infty I^q/I^{q+1},
		$$ 
		on $R/I$, and on $R/L_i$ for all $i$.

		$(i)$ Applying $-\otimes_R S_i$ to the inclusion $\Rees(I) \hookrightarrow R[t]$ yields a natural map 
		$$
		\mathfrak{s} : \Rees(I) \otimes_R S_i \twoheadrightarrow \Rees_{S_i}(J_i) \,\subset \, S_i[t].
		$$
		For $\pp \in \Spec(R) \setminus V(I)$, localizing the surjection $\mathfrak{s}:\Rees(I) \otimes_R S_i \surjects \Rees_{S_i}(J_i)$ at $R\setminus \pp$, we easily see that it becomes an isomorphism.
		It then follows that some power of $I$ annihilates $\Ker({\mathfrak{s}})$, that is, 
		$
		I^l \cdot \Ker(\mathfrak{s})=0
		$
		for some $l>0$.		
		We have that $\dim\left(\gr_I(R)\right)=\dim(R)=d+1$ and $\dim\left(\Rees(I)\right)=\dim(R)+1=d+2$ (see, e.g., \cite[\S 5.1]{huneke2006integral}).
		Therefore, \autoref{lem_reduce_filter_reg} and \cite[Lemma 1.2]{HYRY_MULTIGRAD} yield that 
		\begin{align*}
			\dim\big(\Supp_{++}\left(\Ker(\mathfrak{s})\right)\big) &\,\le\, \dim\big(\Supp_{++}\big((\Rees(I) \otimes_R S_i) \otimes_R (R/I) \big)\big) \\
			&\,=\, \dim\big(\Supp_{++}\left(\gr_I(R) \otimes_R S_i \right)\big) \\
			&\,\le\, \dim(\gr_I(R))-2-i= d-1-i
		\end{align*}
		and 
		$$
		\dim\big(\Supp_{++}\left(\Rees(I) \otimes_R S_i\right)\big) \,=\, \dim\big(\Rees(I)\big)-2-i=d-i.
		$$
		Hence, from the short exact sequence $0 \rightarrow \Ker(\mathfrak{s}) \rightarrow \Rees(I) \otimes_R S_i \rightarrow \Rees_{S_i}(J_i) \rightarrow 0$ and the additivity of multiplicities, it follows that 
		$$
		e(0,d-i;\Rees(I)\otimes_R S_i)=e(0,d-i;\Rees_{S_i}(J_i)).
		$$
		By using \autoref{lem_reduce_filter_reg} successively we obtain
		$$
		d_i(\FF) = e\Big(i,d-i; \Rees(I)\Big)=e\big(0,d-i;\Rees(I) \otimes_R S_i\big).
		$$
		So, the result follows.
		
		$(ii)$ The condition of  $h_1,\ldots,h_d$ being a filter-regular sequence on $R/I$ yields that $J_i$ is an $\mm S_i$-primary ideal for $d-c+1\le i \le d$.
		It then follows that $d_i(\FF) = e\big(0,d-i;\Rees_{S_i}(J_i)\big)=\delta^{d-i}$ (see, e.g., \cite[Observation 3.2]{KPU_blowup_fibers}).
		
		$(iii)$ 
		Since $\pd(R/I)=c$, the Auslander-Buchsbaum formula implies that $\depth(R/I) = d-c$.
		When $R/I$ is Cohen-Macaulay and $\kk$ is infinite, we can assure that $h_1,\ldots,h_{d-c}$ is a regular sequence on $R$ and on $R/I$ (see, e.g., \cite[Proposition 1.5.12]{BRUNS_HERZOG}).
		Then, using that $h_1,\ldots,h_{d-c}$ is a regular sequence on $R$ and on $R/I$, for $1 \le j \le d-c$, it follows that $S_j/J_j\cong R/(I,h_1,\ldots,h_j)$ is Cohen-Macaulay and that 
		$$
		\HH_l\left(F_\bullet\otimes_RS_j\right) \,\cong\, \Tor_l^R(R/I,S_j) \,\cong\, \HH_l\big(K_\bullet(h_1,\ldots,h_j;R/I)\big) \,=\, 0
		$$ 
		for $l\ge 1$ (here $K_\bullet(h_1,\ldots,h_j;R/I)$ denotes the Koszul complex).
		Thus, $F_\bullet\otimes_RS_j$ is the minimal graded free resolution of $S_j/J_j$.
		
		For $1 \le j \le d-c$, since $F_2 \otimes_R S_j \rightarrow F_1 \otimes_R S_j \rightarrow J_j \rightarrow 0$ is a presentation of $J_j$, we get that $\Fitt_i(J_j)=L_iS_j$.
		Since $h_1,\ldots,h_{d-c}$ is a filter-regular sequence on $R/L_i$, the assumption of the condition $G_{d+1}$ yields that $\HT(L_iS_j) = \min\lbrace\HT(L_i),\dim(S_j)\rbrace \ge \min\lbrace i+1, \dim(S_j) \rbrace$ for all $1 \le j \le d-c$.
		This implies that $J_j$ satisfies $G_{d+1-j}$ for all $1 \le j \le d-c$.
		
		So, we are done.
	\end{proof}
\end{proposition}

\subsection{Certain families of rational maps}

In this short subsection we compute all the projective degrees of the rational map $\FF:\PP_\kk^d \dashrightarrow \PP_\kk^n$ when $I$ is a perfect ideal of height two or a Gorenstein ideal of height three, under the assumption of the condition $G_{d+1}$.
The results of this subsection are easy consequences of the previous developments together with \cite[Theorem A]{MULT_SAT_PERF_HT_2} and \cite[Theorem A]{MULT_GOR_HT_3}. 
It should be noted that the condition $G_{d+1}$ is always satisfied by generic perfect ideals of height two and by generic Gorenstein ideals of height three.

\begin{theorem}
	\label{thm_proj_deg_perf_ht_2}
	Assume \autoref{setup_projective_degrees} with the following conditions:
	\begin{enumerate}[(i)]
		\item $I$ is perfect of height two with Hilbert-Burch resolution of the form
		$$
		0 \rightarrow \bigoplus_{i=1}^nR(-\delta-\mu_i) \xrightarrow{\varphi} {R(-\delta)}^{n+1} \rightarrow I\rightarrow 0.			
		$$
		\item $I$ satisfies the condition $G_{d+1}$.
	\end{enumerate}
	Then, the projective degrees of $\FF:\PP_\kk^d \dashrightarrow \PP_\kk^n$ are given by
	$$
	d_i(\FF) =  
	e_{d-i}(\mu_1,\mu_2,\ldots,\mu_n)
	$$
	where $	e_{d-i}(\mu_1,\mu_2,\ldots,\mu_n)$ denotes the elementary symmetric polynomial
	$$
	e_{d-i}(\mu_1,\mu_2,\ldots,\mu_n)= 
	\sum_{1\le j_1 < j_2 < \cdots < j_{d-i} \le n} \mu_{j_1}\mu_{j_2}\cdots\mu_{j_{d-i}}.
	$$
\end{theorem}
\begin{proof}
	We can assume that $\kk$ is an infinite field.
	From \autoref{lem_find_filt_reg_rat_map}, we can find $h_1,\ldots,h_d$ such that, if we set $S_i=R/(h_1,\ldots,h_i)$ and $J_i=IS_i$, then $d_i(\FF) = e\big(0,d-i;\Rees_{S_i}(J_i)\big)$ and, for all $1 \le j \le d-2$, $J_j$ is a perfect ideal of ideal height two that satisfies $G_{d+1-j}$ with syzygies of degrees $\mu_1,\mu_2,\ldots,\mu_n$.
	
	For $0 \le j \le d-2$, note that $d_j(\FF) = e\big(0,d-j;\Rees_{S_i}(J_j)\big)$ is equal to the $0$-th projective degree of a rational map determined by the minimal generators of $J_j$, then \autoref{thm_equal_d_0_sat_fib} and \cite[Theorem A]{MULT_SAT_PERF_HT_2} yield that 
	$$
	d_j(\FF) = e\left(\sF_{S_j}(J_j)\right) = e_{d-j}\left(\mu_1,\ldots,\mu_n\right).
	$$
	On the other hand, the case $d-1 \le j \le d$ follows directly from \autoref{lem_find_filt_reg_rat_map}$(ii)$.
	
	So, we are done.
\end{proof}

\begin{theorem}
		\label{thm_proj_deg_Gor_ht_3}
		Assume \autoref{setup_projective_degrees} with the following conditions:
		\begin{enumerate}[(i)]
			\item $I$ is a Gorenstein ideal of height three.
			\item Every non-zero entry of an alternating minimal presentation matrix of $I$  has degree $D\ge 1$.
			\item $I$ satisfies the condition $G_{d+1}$.		
		\end{enumerate}
	Then, the projective degrees of $\FF:\PP_\kk^d \dashrightarrow \PP_\kk^n$ are given by
	$$
	d_i(\FF) = \begin{cases}
		D^{d-i} \sum_{k=0}^{\lfloor\frac{n-d+i}{2}\rfloor}\binom{n-1-2k}{d-i-1} \quad \text{ if } 0 \le i \le d-3 \\
		\delta^{d-i} \quad\quad\quad\;\,\,\,\,\qquad\qquad\qquad \text{if } d-2 \le i \le d.
	\end{cases}
	$$	
\end{theorem}
\begin{proof}
	The proof follows verbatim to the one of \autoref{thm_proj_deg_perf_ht_2} but now using \cite[Theorem A]{MULT_GOR_HT_3} (instead of using \cite[Theorem A]{MULT_SAT_PERF_HT_2}).
	More explicitly, for $0 \le j \le d-3$, by substituting in the formula of \cite[Theorem A]{MULT_GOR_HT_3} we obtain that 
	$$
	d_j(\FF) \;=\; D^{(d-j+1)-1}\, \sum_{k=0}^{\big\lfloor\frac{(n+1)-(d-j+1)}{2}\big\rfloor}\binom{(n+1)-2-2k}{(d-j+1)-2}.
	$$
	Again, the case $d-2 \le j \le d$ follows directly from \autoref{lem_find_filt_reg_rat_map}$(ii)$.
\end{proof}

\section*{Acknowledgments}
The author is grateful to Mateusz Micha\l{}ek and Bernd Sturmfels for helpful discussions and for sparking the author's interest on the projective degrees of rational maps.
The use of \textit{Macaulay2} \cite{MACAULAY2} was important in the preparation of this paper.

\bibliographystyle{elsarticle-num} 

\begin{bibdiv}
\begin{biblist}

\bib{Bhattacharya}{article}{
      author={Bhattacharya, P.~B.},
       title={The {H}ilbert function of two ideals},
        date={1957},
     journal={Proc. Cambridge Philos. Soc.},
      volume={53},
       pages={568\ndash 575},
}

\bib{Brodmann_Sharp_local_cohom}{book}{
      author={Brodmann, M.~P.},
      author={Sharp, R.~Y.},
       title={Local cohomology.},
     edition={Second},
      series={Cambridge Studies in Advanced Mathematics},
   publisher={Cambridge University Press, Cambridge},
        date={2013},
      volume={136},
        note={An algebraic introduction with geometric applications},
}

\bib{BRUNS_HERZOG}{book}{
      author={Bruns, Winfried},
      author={Herzog, J\"urgen},
       title={Cohen-{M}acaulay rings},
     edition={2},
      series={Cambridge Studies in Advanced Mathematics},
   publisher={Cambridge University Press},
        date={1998},
}

\bib{BuchsbaumEisenbud}{article}{
      author={Buchsbaum, David~A.},
      author={Eisenbud, David},
       title={Algebra structures for finite free resolutions, and some
  structure theorems for ideals of codimension {$3$}},
        date={1977},
        ISSN={0002-9327},
     journal={Amer. J. Math.},
      volume={99},
      number={3},
       pages={447\ndash 485},
}

\bib{MULTPROJ}{article}{
      author={{Bus{\'e}}, Laurent},
      author={{Cid-Ruiz}, Yairon},
      author={{D'Andrea}, Carlos},
       title={{Degree and birationality of multi-graded rational maps}},
        date={2018-05},
     journal={ArXiv e-prints},
      eprint={1805.05180},
}

\bib{MULT_SAT_PERF_HT_2}{article}{
      author={{Cid-Ruiz}, Yairon},
       title={Multiplicity of the saturated special fiber ring of height two
  perfect ideals},
        date={2018},
     journal={to appear in Proc. Amer. Math. Soc.},
        note={1807.03189},
}

\bib{MULT_GOR_HT_3}{article}{
      author={Cid-Ruiz, Yairon},
      author={Mukundan, Vivek},
       title={Multiplicity of the saturated special fiber ring of height three
  gorenstein ideals},
        date={2019},
     journal={arXiv preprint arXiv:1909.13633},
}

\bib{SPECIALIZATION_RAT_MAPS}{article}{
      author={Cid-Ruiz, Yairon},
      author={Simis, Aron},
       title={Degree of rational maps via specialization},
        date={2019},
     journal={arXiv preprint arXiv:1901.06599},
}

\bib{DOLGACHEV}{book}{
      author={Dolgachev, Igor~V.},
       title={Classical algebraic geometry},
   publisher={Cambridge University Press, Cambridge},
        date={2012},
        note={A modern view},
}

\bib{EISEN_COMM}{book}{
      author={Eisenbud, David},
       title={Commutative algebra with a view towards algebraic geometry},
      series={Graduate Texts in Mathematics, 150},
   publisher={Springer-Verlag},
        date={1995},
}

\bib{EISENBUD_HARRIS_SCHEMES}{book}{
      author={Eisenbud, David},
      author={Harris, Joe},
       title={The geometry of schemes},
      series={Graduate Texts in Mathematics},
   publisher={Springer-Verlag, New York},
        date={2000},
      volume={197},
}

\bib{GORTZ_WEDHORN}{book}{
      author={G\"{o}rtz, Ulrich},
      author={Wedhorn, Torsten},
       title={Algebraic geometry {I}},
      series={Advanced Lectures in Mathematics},
   publisher={Vieweg + Teubner, Wiesbaden},
        date={2010},
        ISBN={978-3-8348-0676-5},
         url={https://doi.org/10.1007/978-3-8348-9722-0},
        note={Schemes with examples and exercises},
}

\bib{MACAULAY2}{misc}{
      author={Grayson, Daniel~R.},
      author={Stillman, Michael~E.},
       title={Macaulay2, a software system for research in algebraic geometry},
        note={Available at \url{http://www.math.uiuc.edu/Macaulay2/}},
}

\bib{HARRIS}{book}{
      author={Harris, Joe},
       title={Algebraic geometry},
      series={Graduate Texts in Mathematics},
   publisher={Springer-Verlag, New York},
        date={1995},
      volume={133},
        note={A first course, Corrected reprint of the 1992 original},
}

\bib{HERMANN_MULTIGRAD}{article}{
      author={Herrmann, Manfred},
      author={Hyry, Eero},
      author={Ribbe, J\"{u}rgen},
      author={Tang, Zhongming},
       title={Reduction numbers and multiplicities of multigraded structures},
        date={1997},
     journal={J. Algebra},
      volume={197},
      number={2},
       pages={311\ndash 341},
}

\bib{HERZOG_HIBI_MONOMIALS}{book}{
      author={Herzog, J\"{u}rgen},
      author={Hibi, Takayuki},
       title={Monomial ideals},
      series={Graduate Texts in Mathematics},
   publisher={Springer-Verlag London, Ltd., London},
        date={2011},
      volume={260},
}

\bib{huneke2006integral}{book}{
      author={Huneke, Craig},
      author={Swanson, Irena},
       title={Integral closure of ideals, rings, and modules},
   publisher={Cambridge University Press},
        date={2006},
      volume={13},
}

\bib{HYRY_MULTIGRAD}{article}{
      author={Hyry, Eero},
       title={The diagonal subring and the {C}ohen-{M}acaulay property of a
  multigraded ring},
        date={1999},
     journal={Trans. Amer. Math. Soc.},
      volume={351},
      number={6},
       pages={2213\ndash 2232},
}

\bib{JAYANATHAN_VERMA}{article}{
      author={Jayanthan, A.~V.},
      author={Verma, J.~K.},
       title={Grothendieck-{S}erre formula and bigraded {C}ohen-{M}acaulay
  {R}ees algebras},
        date={2002},
     journal={J. Algebra},
      volume={254},
      number={1},
       pages={1\ndash 20},
}

\bib{VERMA_BIGRAD}{incollection}{
      author={Katz, D.},
      author={Mandal, S.},
      author={Verma, J.~K.},
       title={Hilbert functions of bigraded algebras},
        date={1994},
   booktitle={Commutative algebra ({T}rieste, 1992)},
   publisher={World Sci. Publ., River Edge, NJ},
       pages={291\ndash 302},
}

\bib{Kleiman_geom_mult}{article}{
      author={Kleiman, Steven},
      author={Thorup, Anders},
       title={A geometric theory of the {B}uchsbaum-{R}im multiplicity},
        date={1994},
     journal={J. Algebra},
      volume={167},
      number={1},
       pages={168\ndash 231},
}

\bib{KPU_blowup_fibers}{article}{
      author={Kustin, Andrew},
      author={Polini, Claudia},
      author={Ulrich, Bernd},
       title={Blowups and fibers of morphisms},
        date={2016},
     journal={Nagoya Math. J.},
      volume={224},
      number={1},
       pages={168\ndash 201},
}

\bib{EXPONENTIAL_VARIETIES}{article}{
      author={Micha\l{}ek, Mateusz},
      author={Sturmfels, Bernd},
      author={Uhler, Caroline},
      author={Zwiernik, Piotr},
       title={Exponential varieties},
        date={2016},
     journal={Proc. Lond. Math. Soc. (3)},
      volume={112},
      number={1},
       pages={27\ndash 56},
}

\bib{MILLER_STURMFELS}{book}{
      author={Miller, Ezra},
      author={Sturmfels, Bernd},
       title={Combinatorial commutative algebra},
      series={Graduate Texts in Mathematics},
   publisher={Springer-Verlag, New York},
        date={2005},
      volume={227},
}

\bib{TRUNG_VERMA_SURVEY}{article}{
      author={Trung, N.~V.},
      author={Verma, J.~K.},
       title={Hilbert functions of multigraded algebras, mixed multiplicities
  of ideals and their applications},
        date={2010},
     journal={J. Commut. Algebra},
      volume={2},
      number={4},
       pages={515\ndash 565},
}

\bib{TRUNG_REDUCTIONEXPONENT}{article}{
      author={Trung, Ng\^{o}~Vi\^{e}t},
       title={Reduction exponent and degree bound for the defining equations of
  graded rings},
        date={1987},
     journal={Proc. Amer. Math. Soc.},
      volume={101},
      number={2},
       pages={229\ndash 236},
}

\bib{TRUNG_POSITIVE}{article}{
      author={Trung, Ng\^{o}~Vi\^{e}t},
       title={Positivity of mixed multiplicities},
        date={2001},
     journal={Math. Ann.},
      volume={319},
      number={1},
       pages={33\ndash 63},
}

\bib{VAN_DER_WAERDEN}{inproceedings}{
      author={Van~der Waerden, Bartel~Leendert},
       title={On Hilbert’s function, series of composition of ideals and a
  generalization of the theorem of {B}ezout},
        date={1929},
   booktitle={Proc. roy. acad. amsterdam},
      volume={31},
       pages={749\ndash 770},
}

\end{biblist}
\end{bibdiv}

\end{document}